\documentclass[11pt, reqno]{amsart}
\usepackage{amsmath, amsthm, amscd, amsfonts, amssymb, graphicx, color}
\usepackage[bookmarksnumbered, colorlinks, plainpages]{hyperref}
\makeatletter \oddsidemargin.9375in \evensidemargin \oddsidemargin
\def\authorsaddresses#1{\dedicatory{#1}}
\newtheorem{theorem}{Theorem}[section]
\newtheorem{lemma}[theorem]{Lemma}
\newtheorem{proposition}[theorem]{Proposition}
\newtheorem{corollary}[theorem]{Corollary}
\theoremstyle{definition}
\newtheorem{definition}[theorem]{Definition}
\newtheorem{example}[theorem]{Example}

\theoremstyle{remark}

\theoremstyle{approach}

\numberwithin{equation}{section}

\begin{document}
\setcounter{page}{1}


\title[Woven fusion frames in Hilbert spaces]{Woven fusion frames in Hilbert spaces}

\author[Asghar Rahimi, Zahra Samadzadeh and Bayaz Daraby]{A. Rahimi$ ^{*} $, Z. Samadzadeh$ ^{**} $, B. Daraby$^{\dagger}$ }

\authorsaddresses{$ ^{*,**,\dagger} $Department~of Mathematics,~University~of Maragheh,~Maragheh,~Iran.}
\keywords{Frame; Fusion frame; Perturbation; Riesz decomposition;Woven frame; Weaving frame; Woven fusion frame.
Authors$^{,}$ email:~$ ^{*} $rahimi@maragheh.ac.ir;~$ ^{**} $z.samadzadeh@yahoo.com;~ $ ^{\dagger} $bdaraby@maragheh.ac.ir}
\begin{abstract}
A new notion in frame theory has been introduced recently that called woven frames. 
Woven and weaving frames are powerful tools for pre-processing signals and distributed data processing.
The purpose of introducing fusion frame or frame of subspace is to first construct local components and then build a global frame from these. This type of frame behaves as a generalization of frames. Motivating by the concepts of fusion and weaving frames, we investigate the notion woven-weaving fusion frames and present some of their features. Also, we study some effects of perturbations on woven frames and introduce Riesz decomposition of wovens and then we examine some of behaviors of this type of decomposition.
\end{abstract}
\maketitle

\section{Introduction}

Frames for Hilbert space were first introduced by Duffin and Schaeffer \cite{duffin} to study non-harmonic Fourier series in 1952. After some decades, Daubechies, Grossmann and Meyer reintroduced frames with extensive studies, in 1986 \cite{meyer} and popularized frames from then on. Such that, in the past thirty years, the frame theory became an attractive research and powerful tool for studies such as signal processing, image processing, data compression and sampling theory. 

Frames are generalizations of orthonormal bases in Hilbert spaces. A frame, as well as an orthonormal basis, allows each element in Hilbert space to be written as an infinite linear combination of the frame elements so that unlike the bases conditions, the coefficients might not be unique.

In the early 20'th century, new type of frames were presented to the scientific community, with name of frame of subspaces, which are now known as fusion frames. Fusion frames is a generalization of frames which were introduced by Cassaza and Kutyniok \cite{caku1} in 2003 and were investigated in \cite{caku2, khosravi1, khosravi2, asghari}. The significance of fusion frame is the construction of global frames from local frames in Hilbert space, so the characteristic fusion frame is special suiting for application such as distributed sensing, parallel processing, and packet encoding and so on.

In recent years, Bemrose et.al. introduced weaving frames \cite{wov1}, \cite{wov2}. From the point of view of its introducers, weaving frames are powerful tools for pre-processing signals and distributed data processing.

Improving and extending the notions of fusion and woven (weaving) frames, we investigate the new notion under the name woven (weaving) fusion frames and we prove some new results conserning the consepts fusion and woven.

The paper is organized as follows: Section 2 contains the basic definitions about frames and fusion frames. Section 3 belongs to preliminaries of woven-weving frames, introduces of the woven fusion frames, new notions and applications of them. In Section 4, we study the effects of perturbations on woven fusion frames. Finally, in Section 5, we mention to woven Riesz decomposition and bring some results of them.

\section{frames in Hilbert spaces}
As a preliminary of frames, at the first, we mention discrete frames and fusion frames. Through of this paper, $ \mathbb{I} $ is the indexing set where it can be finite or infinity countable set, $ \mathcal{H} $ is the separable Hilbert space, $ [m] $ is the natural numbers set $ \lbrace1, 2, \ldots, m\rbrace $ and $ P $ is the orthogonal projection.
\subsection{Discrete frame}
We review the definition and some properties of frames. For more information, see \cite{ole}.
\begin{definition}
A countable family of elements $ \left\lbrace f_{i}\right\rbrace _{i\in\mathbb{I} } $ in $ \mathcal{H}$ is a frame for $ \mathcal{H} $, if there exist constants
$ \mathcal{A}, \mathcal{B}>0 $ such that:
\begin{eqnarray*}
\label{1}
\mathcal{A}\Vert f\Vert^{2} \leq\sum_{i\in\mathbb{I}}\vert\left\langle f, f_{i}\right\rangle \vert^{2} \leq\mathcal{B}\Vert f\Vert^{2}, \quad \forall f\in\mathbb{I}.
\end{eqnarray*}
The numbers $ \mathcal{A} $ and $ \mathcal{B} $ are celled the lower and upper frame bounds, respectively. The frame
$ \left\lbrace f_{i}\right\rbrace _{i\in\mathbb{I} } $ is called tight frame, if
$ \mathcal{A}=\mathcal{B} $ and is called Parseval frame, if
$ \mathcal{A}=\mathcal{B}=1 $.
Also the sequence $ \left\lbrace f_{i}\right\rbrace _{i\in\mathbb{I} } $
is called Bessel sequence, if satisfy only the upper inequality. 

If for every  finite scalar sequences $ \left\lbrace c_{i}\right\rbrace _{i\in\mathbb{I} } $, there exist $ 0<\mathcal{A}\leq\mathcal{B}<\infty $ such that:
$$  \mathcal{A}\left( \sum_{i}\vert c_{i}\vert^{2}\right) ^{\frac{1}{2}}\leq \Vert \sum_{i}c_{i}f_{i}\Vert\leq \mathcal{B}\left( \sum_{i}\vert c_{i}\vert^{2}\right) ^{\frac{1}{2}}, $$
and also we have $ \mathcal{H}=\overline{\rm span}\left\lbrace f_{i}\right\rbrace_{i\in\mathbb{I} }  $, the family $ \left\lbrace f_{i}\right\rbrace _{i\in\mathbb{I} } $ is called a Riesz basis.
\end{definition}
Consider now a Hilbert space $ \mathcal{H} $ equipped with a frame
$ \left\lbrace f_{i}\right\rbrace _{i\in\mathbb{I} } $ and define the mapping:
$$ U: \mathcal{H}\longrightarrow\ell^{2}\left(\mathbb{I} \right),\quad U(f)=\left\lbrace \left\langle f, f_{i}\right\rangle \right\rbrace _{i\in\mathbb{I}}. $$
The operator $ U $ is usually called the analysis operator. The adjoint operator is given by:
$$ T:\ell^{2}(\mathbb{I})\longrightarrow\mathcal{H},\quad T\left\lbrace c_{i}\right\rbrace=\sum_{i\in\mathbb{I}}c_{i}f_{i}, $$
and is called the synthesis operator.
Composing $ U $ with its adjoint $ T $, we obtain the frame operator:
$$ S:\mathcal{H}\longrightarrow\mathcal{H},\quad S(f)= TU(f)=\sum_{i\in\mathbb{I}}\left\langle f, f_{i}\right\rangle f_{i}.$$
The operator $ S $ is positive, self-adjoint and invertible and every $ f\in\mathcal{H} $ can be represented as:
$$ f=\sum_{i\in\mathbb{I}}\left\langle f, S^{-1}f_{i}\right\rangle f_{i}=\sum_{i\in\mathbb{I}}\left\langle f, f_{i}\right\rangle S^{-1}f_{i}. $$
The family $ \left\lbrace S^{-1}f_{i}\right\rbrace _{i\in\mathbb{I}} $ is also a frame and it is called the standard dual frame of
$ \left\lbrace f_{i}\right\rbrace _{i\in\mathbb{I}} $.
\subsection{Fusion frame}
In 2003, a new type of generalization of frames were introduced by Cassaza and Kutyniok to the science world that today we know them as fusion frames. In this section, we briefly recall some basic notations, definitions and some important properties of fusion frames that are useful for our study. For more detailed information one can see \cite{asghari, caku1, caku2, khosravi1, khosravi2}.
\begin{definition}
Let $ \left\lbrace \nu_{i}\right\rbrace_{i\in\mathbb{I}}$ be a family of weights such that $ \nu_{i}>0 $ for all $ i\in\mathbb{I} $.
A family of closed subspaces $ \left\lbrace W_{i}\right\rbrace_{i\in\mathbb{I}}$ of a Hilbert space $ \mathcal{H} $ is called a fusion frame (or frame of subspaces) for $ \mathcal{H} $ with respect to weights $ \left\lbrace \nu_{i}\right\rbrace_{i\in\mathbb{I}}$, if there exist constants $ \mathcal{C}, \mathcal{D}>0 $ such that:
\begin{eqnarray}
\label{2}
\mathcal{C}\Vert f\Vert^{2}\leq \sum_{i\in\mathbb{I}}\nu_{i}^{2}\Vert P_{W_{i}}(f)\Vert^{2}\leq \mathcal{D}\Vert f\Vert^{2},
\quad\forall f\in\mathcal{H},
\end{eqnarray}
where $ P_{W_{i}} $ is the orthogonal projection of $ \mathcal{H} $ onto $ W_{i} $.
The constants $ \mathcal{C} $ and $ \mathcal{D} $ are called the lower and upper fusion frame bounds, respectively. If the second inequality in  (\ref{2}) holds, the family of subspace 
$ \left\lbrace W_{i}\right\rbrace_{i\in\mathbb{I}}$ is called a Bessel sequence of subspaces with respect to
$ \left\lbrace \nu_{i}\right\rbrace_{i\in\mathbb{I}}$ with Bessel bound $ \mathcal{D} $.
Also $ \left\lbrace W_{i}\right\rbrace_{i\in\mathbb{I}}$ is called tight fusion frame with respect to
$ \left\lbrace \nu_{i}\right\rbrace_{i\in\mathbb{I}}$, if $ \mathcal{C}=\mathcal{D} $ and is called a Parseval fusion frame if $ \mathcal{C}=\mathcal{D}=1$. We say
$ \left\lbrace W_{i}\right\rbrace_{i\in\mathbb{I}}$ an orthonormal fusion basis for $ \mathcal{H} $, if
$ \mathcal{H}=\bigoplus_{i\in\mathbb{I}}W_{i} $.
\end{definition}

\begin{definition}
The fusion frame $ \left\lbrace W_{i}\right\rbrace _{i\in\mathbb{I} } $ with respect to some family of weights is called a Riesz decomposition of $ \mathcal{H} $, if for every $ f\in\mathcal{H} $, there is a unique choice of $ f_{i}\in W_{i} $ so that $ f=\sum_{i\in\mathbb{I}}f_{i}. $
\end{definition}

\begin{flushleft}
{\LARGE Notation:}
\end{flushleft}
For each family of subspaces $ \left\lbrace W_{i} \right\rbrace _{i\in\mathbb{I}} $ of $ \mathcal{H} $, the representation space:
$$ \left( \sum_{i\in\mathbb{I}}\bigoplus W_{i}\right) _{\ell_{2}}=\left\lbrace \left\lbrace f_{i}\right\rbrace _{i\in\mathbb{I}}\vert f_{i}\in W_{i}\quad and \quad \sum_{i\in \mathbb{I}}\Vert f_{i}\Vert^{2} < \infty\right\rbrace  , $$
with inner product
$$ \left\langle \left\lbrace f_{i}\right\rbrace _{i\in\mathbb{I}}, \left\lbrace g_{i}\right\rbrace _{i\in\mathbb{I}}\right\rangle =\sum_{i\in\mathbb{I}} \left\langle f_{i}, g_{i}\right\rangle $$
is a Hilbert space. This space is needed in the studying of fusion systems.

We require the following lemma to define the analysis, synthesis and fusion frame operators \cite{caku1}.
\begin{lemma}
Let $ \left\lbrace W_{i}\right\rbrace _{i\in\mathbb{I}} $ be a Bessel sequence of subspaces with respect to
$ \left\lbrace \nu_{i}\right\rbrace _{i\in\mathbb{I}} $ for $ \mathcal{H} $. Then for each sequence $ \left\lbrace f_{i}\right\rbrace _{i\in\mathbb{I}} $ with $ f_{i}\in W_{i} $, $ i\in\mathbb{I} $, the series $ \sum_{i\in\mathbb{I}}\nu_{i}f_{i} $ converges unconditionally.
\end{lemma}
\begin{definition}
Let $ \left\lbrace W_{i}\right\rbrace _{i\in\mathbb{I} } $ be a fusion frame for $ \mathcal{H} $ with respect to
$ \left\lbrace \nu_{i}\right\rbrace _{i\in\mathbb{I} } $. Then the analysis operator for $ \left\lbrace W_{i}\right\rbrace _{i\in\mathbb{I} } $ with weights $ \left\lbrace \nu_{i}\right\rbrace _{i\in\mathbb{I} } $ is defined by:
$$ U_{W, \nu }:\mathcal{H}\longrightarrow \left( \sum_{i\in\mathbb{I}}\bigoplus W_{i}\right) _{\ell_{2}}, \quad U_{W,\nu}(f)= \left\lbrace \nu_{i}P_{W_{i}}(f)\right\rbrace _{i\in\mathbb{I} }.$$
The adjoint of $ U_{W,\nu} $ is called the synthesis operator, we denote
$ T_{W,\nu}=U_{W,\nu}^{*} $. By elementary calculation, we have
$$ T_{W,\nu}:\left( \sum_{i\in\mathbb{I}}\bigoplus W_{i}\right) _{\ell_{2}}\longrightarrow\mathcal{H}, \quad T_{W,\nu}\left( \left\lbrace f_{i}\right\rbrace_{i\in\mathbb{I}} \right) =\sum_{i\in\mathbb{I}} \nu_{i}f_{i}, $$
like discrete frames, the fusion frame operator
$ \left\lbrace W_{i}\right\rbrace _{i\in\mathbb{I} } $ with respect to
$ \left\lbrace \nu_{i}\right\rbrace _{i\in\mathbb{I} } $ is the composition of analysis and synthesis operators,
$$ S_{W,\nu}:\mathcal{H}\longrightarrow\mathcal{H}, \quad S_{W,\nu}(f)=T_{W,\nu}U_{W,\nu}(f)=\sum_{i\in\mathbb{I}}\nu_{i}^{2}P_{W_{i}}(f),~~~ \forall f\in\mathcal{H}.$$
\end{definition}
The following theorem present the equivalence conditions between the fusion frames and their operators.
\begin{theorem}
\cite{caku1}
Let $ \left\lbrace W_{i}\right\rbrace _{i\in\mathcal{H}} $ be a family of subspaces in $ \mathcal{H} $ and
$ \left\lbrace \nu_{i}\right\rbrace _{i\in\mathcal{H}} $ be a family of weights. Then the following conditions are equivalent.
\begin{enumerate}
\item[(i)] $ \left\lbrace W_{i}\right\rbrace _{i\in\mathcal{H}} $ is a fusion frame with respect to $ \left\lbrace \nu_{i}\right\rbrace _{i\in\mathcal{H}} $,
\item[(ii)] The synthesis operator $ T_{W,\nu} $ is bounded, linear and onto,
\item[(iii)] The analysis operator $ U_{W, \nu} $ is a (possibly into) isomorphism.
\end{enumerate}
\end{theorem}
\section{Woven frame}
Woven frames in Hilbert spaces, were introduced in 2015 by Bemrose et al. \cite{wov1, wov2, wov3}, after that, Vashisht,  Deepshikha, and etc. have done more research \cite{vd2, vd3, vd4, vd5}. They have studied a variety of different types of generalized weaving frames, such as g-frame, K-frame, and continuous frame.  In the following, we mention the definition of woven frames with an example.
\begin{definition}
 Let $ F= \left\lbrace f_{ij}\right\rbrace _{i\in\mathbb{I} } $ for $ j\in\left[ m\right] $ be a family of frames for separable Hilbert space
 $ \mathcal{H} $. If there exist universal constants $ \mathcal{A^{\prime}}$ and $\mathcal{B^{\prime}} $, such that for every partition
$ \left\lbrace \sigma_{j}\right\rbrace _{j\in\left[ m\right] } $, the family
$ F_{j}=\left\lbrace f_{ij}\right\rbrace _{i\in\sigma_{j}} $ is a frame for $ \mathcal{H} $ with bounds
$ \mathcal{A^{\prime}}$ and $\mathcal{B^{\prime}} $, then $ F $ is said  woven frames and for every
$ j\in\left[ m\right] $ and the frames $ F_{j} $ are called weaving frame.
\end{definition}
Now, we introduce two frames in the Euclidean space that form woven frames.
 \begin{example}
Let $ \left\lbrace e_{i}\right\rbrace_{i=1}^{2} $ be the standard basis for Euclidean space $ \mathbb{R}^{2} $.
Let $ F $ and $ G $ be the sets:
$$ F=\left\lbrace f_{i}\right\rbrace _{i=1}^{3}=\left\lbrace 2e_{2}, 3e_{1}, 2e_{1}+3e_{2}\right\rbrace  $$
and
$$ G=\left\lbrace g_{i}\right\rbrace _{i=1}^{3}=\left\lbrace e_{1}, e_{2}, 3e_{1}+e_{2}\right\rbrace . $$
$ F $ and $ G $ are frames for Euclidean space $ \mathbb{R}^{2} $. For any $ f\in\mathbb{R}^{2} $
$$ \sum_{i=1}^{3}\vert\left\langle f,f_{i}\right\rangle \vert^{2}=\vert\left\langle f, f_{1}\right\rangle\vert^{2} +\vert\left\langle f, f_{2}\right\rangle\vert^{2} + \vert\left\langle f, f_{3}\right\rangle\vert^{2} , $$
 therefor we have
$$ 4\Vert f\Vert^{2}\leq\sum_{i=1}^{3}\vert\left\langle f,f_{i}\right\rangle \vert^{2}\leq 22\Vert f\Vert^{2}, \quad \forall f\in\mathbb{R}^{2}.$$
So $ F $ is a frame with lower and upper bounds $ 4 $ and $ 22 $, respectively. It is important to note that, these bounds may not be optimal. Similarly, $ G $ is a frame with bounds $ 1 $ and $ 19 $. The frames $ F $ and $ G $ constitute a woven frame. For example, if we assume that
$ \sigma_{1} =\left\lbrace 1,2 \right\rbrace $, then for any $ f $
\begin{eqnarray*}
\sum_{i\in\sigma_{1}}\vert\left\langle f,f_{i}\right\rangle \vert^{2}+\sum_{i\in\sigma_{1}^{c}}\vert\left\langle f,g_{i}\right\rangle \vert^{2}
= \vert\left\langle f,f_{1}\right\rangle \vert^{2}+\vert\left\langle f,f_{2}\right\rangle \vert^{2}+ \vert\left\langle f,g_{3}\right\rangle \vert^{2},
\end{eqnarray*}
therefor we have
\begin{eqnarray*}
4\Vert f\Vert^{2}\leq \sum_{i\in\sigma_{1}}\vert\left\langle f,f_{i}\right\rangle \vert^{2}+
\sum_{i\in\sigma_{1}^{c}}\vert\left\langle f,g_{i}\right\rangle \vert^{2}\leq 27 \Vert f\Vert^{2}.
\end{eqnarray*}
So $ \left\lbrace f_{i}\right\rbrace _{i\in\sigma_{1}}\bigcup\left\lbrace g_{i}\right\rbrace _{i\in\sigma_{1}^{c} }$
is frame with lower and upper bounds $ \mathcal{A}^{\prime}_{1}=4 $ and
$ \mathcal{B}^{\prime}_{1}=27 $, respectively. Similarly for every
$ \sigma_{j}\subset\left\lbrace 1,2,3\right\rbrace  $, for
$ 1\leq j \leq 8 $,
$ \left\lbrace f_{i}\right\rbrace _{i\in\sigma_{j}}\bigcup\left\lbrace g_{i}\right\rbrace _{i\in\sigma_{j}^{c}} $ is frame. Then
$ \left\lbrace f_{i}\right\rbrace _{i=1}^{3} $ and $ \left\lbrace g_{i}\right\rbrace _{i=1}^{3} $ are woven frames with universal bounds
$ \mathcal{A}^{\prime}=\min_{1\leq j\leq 8} \mathcal{A}^{\prime}_{j} $ and
$ \mathcal{B}^{\prime}=\min_{1\leq j\leq 8}\mathcal{B}^{\prime}_{j} $.
 \end{example}
The following theorem shows that woven frames and invariant under a bounded are invertible operator with different bounds.
\begin{theorem}
\label{T3}
Let $ \left\lbrace f_{ij}\right\rbrace _{i\in\mathbb{I},j\in\left[ m\right]  } $ be woven frame for $ \mathcal{H} $ with universal bounds
$ \mathcal{A}^{\prime} $ and $ \mathcal{B}^{\prime} $. If $ E $ is bounded and invertible operator on $ \mathcal{H} $,
then $ \left\lbrace Ef_{ij}\right\rbrace _{i\in\mathbb{I},j\in\left[ m\right]  } $ is woven frame for $ \mathcal{H} $ with universal bounds
$ \mathcal{A}^{\prime}\left\Vert E^{-1}\right\Vert^{-2} $ and $ \mathcal{B}^{\prime}\left\Vert E\right\Vert^{2} $.
\end{theorem}
\begin{proof}
Since $ \left\lbrace f_{ij}\right\rbrace _{i\in\mathbb{I},j\in\left[ m\right]  } $ is a woven frame for $ \mathcal{H} $, then for every
$ \sigma_{j}\subset\mathbb{I}$, $ j\in\left[ m\right] $, the sequence $\left\lbrace f_{ij}\right\rbrace _{i\in\sigma_{j},j\in\left[ m\right]  }$ is a frame with bounds $ \mathcal{A}^{\prime} $ and $ \mathcal{B}^{\prime} $. The boundedness of $ E $ verifies the upper bound
$$\sum_{i\in\sigma_{j}, j\in[m]}\left\vert\left\langle f,Ef_{ij}\right\rangle \right\vert^{2}\leq\mathcal{B}^{\prime}\left\Vert E^{\ast}f\right\Vert^{2}\leq
\mathcal{B}^{\prime}\left\Vert E^{\ast}\right\Vert^{2} \left\Vert f\right\Vert^{2}=
\mathcal{B}^{\prime}\left\Vert E\right\Vert^{2} \left\Vert f\right\Vert^{2}.$$
For lower bound, we assume that $ g\in\mathcal{H} $. Since $ E $ is surjective, there exist $ f\in\mathcal{H} $ such that
$ Ef=g $. Therefor we have:
\begin{eqnarray*}
\left\Vert g\right\Vert^{2}&=&\left\Vert Ef\right\Vert^{2}\\
&=&\left\Vert \left( EE^{-1}\right) ^{*}Ef\right\Vert^{2}\\
&=&\left\Vert \left( E^{-1}\right) ^{*}E^{*}Ef\right\Vert^{2}\\
&\leq & \left\Vert E^{-1}\right\Vert^{2}\left\Vert E^{*}Ef\right\Vert^{2}\\
&\leq & \frac{\left\Vert E^{-1}\right\Vert^{2}}{\mathcal{A}^{\prime}}\sum_{i\in\sigma_{j} ,j\in [m]}
\left\vert \left\langle E^{*}Ef,f_{ij}\right\rangle \right\vert^{2}\\
&=& \frac{\left\Vert E^{-1}\right\Vert^{2}}{\mathcal{A}^{\prime}}\sum_{i\in\sigma_{j} ,j\in [m]}
\left\vert \left\langle Ef,Ef_{ij}\right\rangle \right\vert^{2}.
\end{eqnarray*}
So:
$$  \mathcal{A}^{\prime}\left\Vert E^{-1}\right\Vert^{-2}\left\Vert g\right\Vert^{2}\leq\sum_{i\in\sigma_{j},j\in [m]}
\left\vert \left\langle Ef,Ef_{ij}\right\rangle \right\vert^{2}. $$
\end{proof}

\subsection{Woven fusion frames}
Extending and improving the notions of fusion and weaving frames, we introduce the woven fusion frames and we show that the equivalence of discrete frames and bases with woven fusion frames and examin effects of operators on those. Also, we present some results for this type of frames in the examples.
\begin{definition}
A family of fusion frames $ \left\lbrace W_{ij}\right\rbrace _{i=1}^{\infty}$, for $ j\in\left[ m \right] $
, with respect to weights
$ \lbrace\nu_{ij}\rbrace_{i\in \mathbb{I}, j\in\left[ m\right] }  $, is said woven fusion frames if there are universal constant $ \mathcal{A} $ and $ \mathcal{B} $, such that for every partition
$ \lbrace \sigma_{j}\rbrace_{j\in\left[ m \right] } $ of $ \mathbb{I} $ , the family
$ \lbrace W_{ij}\rbrace_{i\in \sigma_{j}, j\in\left[ m\right]  } $ is a fusion frame for $ \mathcal{H} $ with lower and upper frame bounds $ \mathcal{A} $ and $ \mathcal{B} $. Each family $ \lbrace W_{ij}\rbrace_{i\in \sigma_{j}, j\in\left[ m\right]  } $ is called a weaving fusion frame.
\end{definition}

For abrivation, we use W.F.F instead of the statement of woven fusion frame. Also, note that through of this paper, the sequence 
$ \left\lbrace f_{i,j}\right\rbrace_{i,j}  $
is different from the family of sequences 
$ \left\lbrace f_{ij}\right\rbrace_{i,j}  $ in the definition of woven frames.

The following theorem states the equivalence conditions between woven frames and woven fusion frames (W.F.F).

\begin{theorem}
\label{T1}
Suppose for every $ i\in \mathbb{I} $, $ \mathbb{J}_{i} $ is a subset of the index set $ \mathbb{I} $ and $ \nu_{i}, \mu_{i}>0 $. Let $ \lbrace f_{i,j}\rbrace_{j\in \mathbb{J}_{i}} $ and $ \lbrace g_{i,j}\rbrace_{j\in \mathbb{J}_{i}} $ be frame sequences in $ \mathcal{H} $ with frame bounds $ \left( \mathcal{A}_{f_{i}}, \mathcal{B}_{f_{i}}\right)  $ and $ \left( \mathcal{A}_{g_{i}}, \mathcal{B}_{g_{i}}\right)  $ respectively. Define
$$ W_{i}=\overline{\rm span} \left\lbrace f_{i,j}\right\rbrace_{j\in \mathbb{J}_{i}}  , \quad   V_{i}=\overline{{\rm span}} \left\lbrace g_{i,j}\right\rbrace_{j\in \mathbb{J}_{i}} ,\quad \forall i\in \mathbb{I}, $$
and choose orthonormal bases $ \lbrace e_{i,j}\rbrace_{j\in \mathbb{J}_{i}} $ and $ \lbrace e^{\prime}_{i,j}\rbrace_{j\in \mathbb{J}_{i}} $ for each subspaces $ W_{i} $ and $ V_{i} $, respectively. Suppose that
$$ 0<\mathcal{A}_{f}= \inf_{i\in \mathbb{I}}\mathcal{A}_{f_{i}} \leq \mathcal{B}_{f}= \sup_{i\in\mathbb{I}}\mathcal{B}_{g_{i}}< \infty $$
and
$$ 0<\mathcal{A}_{g}= \inf_{i\in \mathbb{I}}\mathcal{A}_{f_{i}} \leq \mathcal{B}_{g}= \sup_{i\in\mathbb{I}}\mathcal{B}_{g_{i}}<\infty . $$
Then the following conditions are equivalent:
\begin{enumerate}
\item[(i)] $ \left\lbrace \nu_{i}f_{i,j}\right\rbrace _{i\in \mathbb{I}, j\in \mathbb{J}_{i}} $ and $ \left\lbrace \mu_{i}g_{i,j}\right\rbrace  _{i\in \mathbb{I}, j\in \mathbb{J}_{i}} $ are woven frames in $ \mathcal{H} $.
\item[(ii)] $ \left\lbrace \nu_{i}e_{i,j}\right\rbrace _{i\in \mathbb{I}, j\in \mathbb{J}_{i}} $ and $ \left\lbrace \mu_{i}e^{\prime}_{i,j}\right\rbrace  _{i\in \mathbb{I}, j\in \mathbb{J}_{i}} $ are woven frames in $ \mathcal{H} $.
\item[(iii)] $ \left\lbrace W_{i}\right\rbrace _{i\in \mathbb{I} } $ and $ \left\lbrace V_{i}\right\rbrace _{i\in \mathbb{I} } $ are $\rm W.F.F$ in $ \mathcal{H} $ with respect to weights $ \left\lbrace \nu_{i}\right\rbrace _{i\in \mathbb{I}} $ , $ \left\lbrace \mu_{i}\right\rbrace _{i\in \mathbb{I}} $, respectively.
\end{enumerate}
\end{theorem}
\begin{proof}
Since for every $ i\in\mathbb{I}$, $\left\lbrace f_{i,j}\right\rbrace _{j\in\mathbb{J}_{i}}$ and
$ \left\lbrace g_{i,j}\right\rbrace _{j\in\mathbb{J}_{i}}$ are frames for $ W_{i}$ and $ V_{i} $ with frame bounds $ \left( \mathcal{A}_{f_{i}}, \mathcal{B}_{f_{i}}\right)  $ and $ \left( \mathcal{A}_{g_{i}}, \mathcal{B}_{g_{i}}\right) $, then for $ \sigma\subset\mathbb{I} $;
\begin{eqnarray*}
&&\mathcal{A}_{f}\sum_{i\in\sigma}\nu_{i}^{2}\Vert P_{W_{i}}\left( f\right) \Vert^{2} +\mathcal{A}_{g}\sum_{i\in\sigma^{c}}\mu_{i}^{2}\Vert P_{V_{i}}\left( f\right) \Vert^{2}\\&\leq & \sum_{i\in\sigma}\mathcal{A}_{f_{i}}\nu_{i}^{2}\Vert P_{W_{i}}\left( f\right) \Vert^{2} +\sum_{i\in\sigma^{c}}\mathcal{A}_{g_{i}}\mu_{i}^{2}\Vert P_{V_{i}}\left( f\right) \Vert^{2}\\&=& \sum_{i\in\sigma}\mathcal{A}_{f_{i}}\Vert\nu_{i} P_{W_{i}}\left( f\right) \Vert^{2} +\sum_{i\in\sigma^{c}}\mathcal{A}_{g_{i}}\Vert\mu_{i} P_{V_{i}}\left( f\right) \Vert^{2}\\&\leq &\sum_{i\in\sigma}\sum_{j\in\mathbb{J}_{i} }\vert\langle\nu_{i}P_{W_{i}}\left( f\right), f_{i,j}\rangle\vert^{2}+\sum_{i\in\sigma^{c}}\sum_{j\in\mathbb{J}_{i} }\vert\langle\mu_{i}P_{V_{i}}\left( f\right), g_{i,j}\rangle\vert^{2}\\&\leq &\sum_{i\in\sigma}\mathcal{B}_{f_{i}}\Vert\nu_{i} P_{W_{i}}\left( f\right) \Vert^{2} +\sum_{i\in\sigma^{c}}\mathcal{B}_{g_{i}}\Vert\mu_{i} P_{V_{i}}\left( f\right) \Vert^{2}\\&\leq &\mathcal{B}_{f}\sum_{i\in\sigma}\Vert\nu_{i} P_{W_{i}}\left( f\right) \Vert^{2} +\mathcal{B}_{g}\sum_{i\in\sigma^{c}}\Vert\mu_{i} P_{V_{i}}\left( f\right) \Vert^{2}.
\end{eqnarray*}
${\rm ( i)}\Rightarrow {\rm (iii)} $: Let $ \left\lbrace \nu_{i}f_{i,j}\right\rbrace _{i\in \mathbb{I}, j\in \mathbb{J}_{i}} $ and $ \left\lbrace \mu_{i}g_{i,j}\right\rbrace  _{i\in \mathbb{I}, j\in \mathbb{J}_{i}} $ be woven frame for $ \mathcal{H} $, with universal frame bounds $ \mathcal{C} $ and $ \mathcal{D} $. The above calculation shows that for every $ f\in \mathcal{H} $,
\begin{eqnarray*}
&&\sum_{i\in\sigma}\nu_{i}^{2}\Vert P_{W_{i}}\left( f\right) \Vert^{2} +\sum_{i\in\sigma^{c}}\mu_{i}^{2}\Vert P_{V_{i}}\left( f\right) \Vert^{2}\\
&\leq &\dfrac{1}{\mathcal{A}}\left( \sum_{i\in\sigma}\sum_{j\in\mathbb{J}_{i} }\vert\langle P_{W_{i}}\left( f\right), \nu_{i}f_{i,j}\rangle\vert^{2}+\sum_{i\in\sigma^{c}}\sum_{j\in\mathbb{J}_{i} }\vert\langle P_{V_{i}}\left( f\right), \mu_{i}g_{i,j}\rangle\vert^{2}\right) \\
&=&\dfrac{1}{\mathcal{A}}\left( \sum_{i\in\sigma}\sum_{j\in\mathbb{J}_{i} }\vert\langle f, \nu_{i}f_{i,j}\rangle\vert^{2}+\sum_{i\in\sigma^{c}}\sum_{j\in\mathbb{J}_{i} }\vert\langle f, \mu_{i}g_{i,j}\rangle\vert^{2}\right) \\
&\leq &\dfrac{\mathcal{D}}{\mathcal{A}}\Vert f\Vert^{2},
\end{eqnarray*}
where $ \mathcal{A}=\min\left\lbrace \mathcal{A}_{f}, \mathcal{A}_{g}\right\rbrace $. For lower frame bound,
\begin{eqnarray*}
&&\sum_{i\in\sigma}\nu_{i}^{2}\Vert P_{W_{i}}\left( f\right) \Vert^{2} +\sum_{i\in\sigma^{c}}\mu_{i}^{2}\Vert P_{V_{i}}\left( f\right) \Vert^{2}\\
&\geq &\dfrac{1}{\mathcal{B}}\left( \sum_{i\in\sigma}\sum_{j\in\mathbb{J}_{i} }\vert\langle P_{W_{i}}\left( f\right), \nu_{i}f_{i,j}\rangle\vert^{2}+\sum_{i\in\sigma^{c}}\sum_{j\in\mathbb{J}_{i} }\vert\langle P_{V_{i}}\left( f\right), \mu_{i}g_{i,j}\rangle\vert^{2}\right) \\
&=&\dfrac{1}{\mathcal{B}}\left( \sum_{i\in\sigma}\sum_{j\in\mathbb{J}_{i} }\vert\langle f, \nu_{i}f_{i,j}\rangle\vert^{2}+\sum_{i\in\sigma^{c}}\sum_{j\in\mathbb{J}_{i} }\vert\langle f, \mu_{i}g_{i,j}\rangle\vert^{2}\right) \\
&\geq &\dfrac{\mathcal{C}}{\mathcal{B}}\Vert f\Vert^{2},
\end{eqnarray*}
for every $ f\in\mathcal{H} $, $ \mathcal{B}=\max\left\lbrace \mathcal{B}_{f}, \mathcal{B}_{g}\right\rbrace $. This calculations consequences 
$\rm (iii)$.

$\rm (iii)\Rightarrow\rm (i) $: Let $ \left\lbrace W_{i}\right\rbrace _{i\in \mathbb{I} } $ and
$ \left\lbrace V_{i}\right\rbrace _{i\in \mathbb{I} } $ be W.F.F with universal frame bounds $ \mathcal{C} $ and
$ \mathcal{D} $. Then for every $ f\in\mathcal{H} $, we have
\begin{eqnarray*}
&&\sum_{i\in\sigma}\sum_{j\in\mathbb{J}_{i} }\vert\langle f, \nu_{i}f_{i,j}\rangle\vert^{2}+\sum_{i\in\sigma^{c}}\sum_{j\in\mathbb{J}_{i} }\vert\langle f, \mu_{i}g_{i,j}\rangle\vert^{2}\\
&=& \sum_{i\in\sigma}\sum_{j\in\mathbb{J}_{i} }\vert\langle\nu_{i} P_{W_{i}}\left( f\right), f_{i,j}\rangle\vert^{2}+\sum_{i\in\sigma^{c}}\sum_{j\in\mathbb{J}_{i} }\vert\langle\mu_{i} P_{V_{i}}\left( f\right), g_{i,j}\rangle\vert^{2}\\
&\geq & \sum_{i\in\sigma}\mathcal{A}_{f_{i}}\nu_{i}^{2}\Vert P_{W_{i}}\left( f\right) \Vert^{2} +\sum_{i\in\sigma^{c}}\mathcal{A}_{g_{i}}\mu_{i}^{2}\Vert P_{V_{i}}\left( f\right) \Vert^{2}\\
&\geq &\mathcal{A}\left( \sum_{i\in\sigma}\nu_{i}^{2}\Vert P_{W_{i}}\left( f\right) \Vert^{2} +\sum_{i\in\sigma^{c}}\mu_{i}^{2}\Vert P_{V_{i}}\left( f\right) \Vert^{2}\right) \\
&\geq &\mathcal{A}\mathcal{C}\Vert f\Vert^{2},
\end{eqnarray*}
and similarly
$$ \sum_{i\in\sigma}\sum_{j\in\mathbb{J}_{i} }\vert\langle f, \nu_{i}f_{i,j}\rangle\vert^{2}+\sum_{i\in\sigma^{c}}\sum_{j\in\mathbb{J}_{i} }\vert\langle f, \mu_{i}g_{i,j}\rangle\vert^{2}\leq\mathcal{B}\mathcal{D}\Vert f\Vert^{2} . $$
So $\rm (i)$ holds.

$\rm (ii)\Leftrightarrow\rm (iii) $: Since $ \left\lbrace e_{i,j}\right\rbrace _{j\in\mathbb{J}_{i} } $ and $ \left\lbrace e^{\prime}_{i,j}\right\rbrace _{j\in\mathbb{J}_{i} } $ are orthonormal bases for subspaces $ W_{i} $ and $ V_{i} $, respectively, then for any $ f\in\mathcal{H} $, we have:
\begin{eqnarray*}
&& \sum_{i\in\sigma}\nu_{i}^{2}\Vert P_{W_{i}}\left( f\right) \Vert^{2} +\sum_{i\in\sigma^{c}}\mu_{i}^{2}\Vert P_{V_{i}}\left( f\right) \Vert^{2}\\
&=&\sum_{i\in\sigma}\nu_{i}^{2}\Vert \sum_{j\in\mathbb{J}}\langle f, e_{i,j}\rangle e_{i,j} \Vert^{2} +\sum_{i\in\sigma^{c}}\mu_{i}^{2}\Vert \sum_{j\in\mathbb{J}}\langle f, e^{\prime}_{i,j}\rangle e^{\prime}_{i,j} \Vert^{2}\\
&=& \sum_{i\in\sigma}\nu_{i}^{2} \sum_{j\in\mathbb{J}}\vert\langle f, e_{i,j}\rangle\vert^{2} +\sum_{i\in\sigma^{c}}\mu_{i}^{2} \sum_{j\in\mathbb{J}}\vert\langle f, e^{\prime}_{i,j} \rangle\vert^{2}\\
&=& \sum_{i\in\sigma} \sum_{j\in\mathbb{J}}\vert\langle f, \nu_{i}e_{i,j}\rangle\vert^{2} +\sum_{i\in\sigma^{c}} \sum_{j\in\mathbb{J}}\vert\langle f, \mu_{i} e^{\prime}_{i,j} \rangle \vert^{2}.
\end{eqnarray*}
So $\rm (ii)$ is equivalent with $\rm (iii)$.
\end{proof}
Combining of Theorem \ref{T3} and Theorem \ref{T1}, we get the following result.
\begin{theorem}
Assume that $ \left\lbrace W_{i}\right\rbrace _{i\in\mathbb{I} } $ and $ \left\lbrace V_{i}\right\rbrace _{i\in\mathbb{I} } $
are fusion frames with weights
$ \left\lbrace \mu_{i}\right\rbrace _{i\in\mathbb{I} } $ and $ \left\lbrace \nu_{i}\right\rbrace _{i\in\mathbb{I} } $
respectively. Also, if $ \left\lbrace W_{i}\right\rbrace _{i\in\mathbb{I} } $ and $ \left\lbrace V_{i}\right\rbrace _{i\in\mathbb{I} } $
are $\rm W.F.F$ and $ E $ is a self-adjoint and invertible operator on $ \mathcal{H} $, such that
$ E^{*}E(W)\subset W $, for every closed subspace $ W $ of $ \mathcal{H} $. Then for every $ \sigma\subset\mathbb{I} $, the sequence
$ \left\lbrace EW_{i}\right\rbrace _{i\in\sigma}\bigcup \left\lbrace EV_{i}\right\rbrace _{i\in\sigma^{c}} $ is a fusion frame with frame operator $ ES_{\sigma}E^{-1} $ where $ S_{\sigma} $ is frame operator of
$ \left\lbrace EW_{i}\right\rbrace _{i\in\sigma}\bigcup \left\lbrace EV_{i}\right\rbrace _{i\in\sigma^{c}} $, i.e.
$ \left\lbrace EW_{i}\right\rbrace _{i\in\mathbb{I} } $ and $ \left\lbrace EV_{i}\right\rbrace _{i\in\mathbb{I} } $ are $\rm W.F.F$.
\end{theorem}
\begin{proof}
Let $ F_{i}=\left\lbrace f_{i,j}\right\rbrace _{j\in\mathbb{J}_{i} } $ and $ G_{i}=\left\lbrace g_{i,j}\right\rbrace _{j\in\mathbb{J}_{i} } $
be frames for $ W_{i} $ and $ V_{i} $ with frame operator $ S_{F_{i}} $ and $ S_{G_{i}} $, respectively. Therefore
$ \left\lbrace Ef_{i,j}\right\rbrace _{j\in\mathbb{J}_{i} } $ is frame for $ EW_{i} $, with frame operator $ ES_{F_{i}}E $:
\begin{eqnarray*}
\sum_{j\in\mathbb{J}_{i}}\left\langle f, Ef_{i,j}\right\rangle Ef_{i,j}&=&
E\left( \sum_{j\in\mathbb{J}_{i}}\left\langle E^{*}\vert_{W_{i}}f, f_{i,j}\right\rangle f_{i,j}\right) \\
&=& E\left( \sum_{j\in\mathbb{J}_{i}}\left\langle Ef, f_{i,j}\right\rangle f_{i,j}\right) \\
&=&ES_{F_{i}}Ef,\quad \forall f\in W_{i}.
\end{eqnarray*}
The standard dual frame of $ \left\lbrace Ef_{i,j}\right\rbrace _{j\in\mathbb{J}_{i} } $ is
$$ \left\lbrace \left( ES_{F_{i}}E\right) ^{-1}Ef_{i,j}\right\rbrace _{j\in\mathbb{J}_{i} }=
\left\lbrace E^{-1}S_{F_{i}}f_{i,j}\right\rbrace_{j\in\mathbb{J}_{i} } . $$
Also $ \left\lbrace Eg_{i,j}\right\rbrace _{j\in\mathbb{J}_{i} } $ is frame for $ EV_{i} $ with frame operator $ ES_{G_{i}}E $ and
standard dual frame $ \left\lbrace E^{-1}S_{G_{i}}f_{i,j}\right\rbrace_{j\in\mathbb{J}_{i} } $. Thus, for
$ \sigma\subset\mathbb{I} $ and by definition of fusion frame operator, for any $ f\in\mathcal{H} $, we have
\begin{eqnarray*}
& & \sum_{i\in\sigma}\nu_{i}^{2}P_{EW_{i}}(f)+\sum_{i\in\sigma^{c}}\mu_{i}^{2}P_{EV_{i}}(f)\\
&=&\sum_{i\in\sigma}\nu_{i}^{2}\left( \sum_{j\in\mathbb{J}_{i} }\left\langle f,E^{-1}S_{F_{i}}f_{i,j}\right\rangle Ef_{i,j}\right) +
\sum_{i\in\sigma^{c}}\mu_{i}^{2}\left( \sum_{j\in\mathbb{J}_{i} }\left\langle f,E^{-1}S_{G_{i}}g_{i,j}\right\rangle Tg_{i,j}\right)\\
&=& E\left(\sum_{i\in\sigma}\nu_{i}^{2}\sum_{j\in\mathbb{J}_{i} }\left\langle S_{F_{i}}E^{-1}f,f_{i,j}\right\rangle f_{i,j}\right)+
E\left(\sum_{i\in\sigma^{c}}\mu_{i}^{2}\sum_{j\in\mathbb{J}_{i} }\left\langle S_{G_{i}}E^{-1}f,g_{i,j}\right\rangle g_{i,j}\right)\\
&=& ES_{\sigma}\left( \sum_{i\in\sigma,j\in\mathbb{J}_{i}}\left\langle E^{-1}f,\nu_{i}f_{i,j}\right\rangle \nu_{i}f_{i,j}+
\sum_{i\in\sigma^{c},j\in\mathbb{J}_{i}}\left\langle E^{-1}f,\mu_{i}g_{i,j}\right\rangle \mu_{i}g_{i,j}\right) \\
&=& ES_{\sigma}E^{-1}f,
\end{eqnarray*}
therefore:
$$  S_{\sigma}f=\sum_{i\in\sigma,j\in\mathbb{J}_{i}}\left\langle f,\nu_{i}f_{i,j}\right\rangle \nu_{i}f_{i,j}+ \sum_{i\in\sigma^{c},j\in\mathbb{J}_{i}}\left\langle f,\mu_{i}g_{i,j}\right\rangle \mu_{i}g_{i,j}. $$
Since $ \left\lbrace W_{i}\right\rbrace _{i\in\mathbb{I} } $ and
$ \left\lbrace V_{i}\right\rbrace _{i\in\mathbb{I} } $
are W.F.F, then by Theorem \ref{T1},
$ \left\lbrace f_{i,j}\right\rbrace _{j\in\mathbb{J}_{i} } $ and
$ \left\lbrace g_{i,j}\right\rbrace _{j\in\mathbb{J}_{i} } $ are woven frames. By Theorem \ref{T3},
$ \left\lbrace Ef_{i,j}\right\rbrace _{j\in\mathbb{J}_{i} } $ and $ \left\lbrace Eg_{i,j}\right\rbrace _{j\in\mathbb{J}_{i} } $ are also woven frames. Thus for every $ f\in\mathcal{H} $ and for arbitrary $ \sigma_{i}\subset\mathbb{I} $, we have:
\begin{eqnarray*}
&&\sum_{i\in\sigma_{i}}\nu_{i}^{2}\left\Vert P_{EW_{i}}(f)\right\Vert^{2}+
\sum_{i\in\sigma_{i}^{c}}\mu_{i}^{2}\left\Vert P_{EV_{i}}(f)\right\Vert^{2}\\
&=& \sum_{i\in\sigma_{i}}\nu_{i}^{2}\sum_{j\in\mathbb{J}_{i} }
\left\vert \left\langle P_{EW_{i}}(f),Ef_{i,j}\right\rangle \right\vert^{2}+
\sum_{i\in\sigma_{i}^{c}}\mu_{i}^{2}\sum_{j\in\mathbb{J}_{i} }
\left\vert \left\langle P_{EV_{i}}(f),Eg_{i,j}\right\rangle \right\vert^{2}\\
&=&\sum_{j\in\mathbb{J}_{i},i\in\sigma_{i} }\left\vert \left\langle P_{EW_{i}}(f),\nu_{i}Ef_{i,j}\right\rangle \right\vert^{2}+
\sum_{j\in\mathbb{J}_{i},i\in\sigma_{i}^{c} }\left\vert \left\langle P_{EV_{i}}(f),\mu_{i}Eg_{i,j}\right\rangle \right\vert^{2}.
\end{eqnarray*}
Therefore $ \left\lbrace EW_{i}\right\rbrace _{i\in\sigma_{i}} \bigcup\left\lbrace EV_{i}\right\rbrace _{i\in\sigma^{c}_{i}} $ is a fusion frame and $ \left\lbrace EW_{i}\right\rbrace _{i\in\mathbb{I}} $ and $ \left\lbrace EV_{i}\right\rbrace _{i\in\mathbb{I}} $ are W.F.F .
\end{proof}

In the following theorem, we show that the intersection of components of a woven fusion frames with the other subspace, is a woven fusion frames (W.F.F) for smaller space.
\begin{theorem}
Let $ K $ be a closed subspace of $ \mathcal{H} $ and let $ \left\lbrace W_{i}\right\rbrace _{i\in \mathbb{I} } $ and $ \left\lbrace V_{i}\right\rbrace _{i\in \mathbb{I} } $ constitute $\rm W.F.F$ with respect to weights $ \left\lbrace \nu_{i}\right\rbrace _{i\in \mathbb{I} } $ and $ \left\lbrace \mu_{i}\right\rbrace _{i\in \mathbb{I} } $ for $ \mathcal{H} $ with woven bounds $ \mathcal{A} $ and $ \mathcal{B} $. Then $ \left\lbrace W_{i}\bigcap K\right\rbrace _{i\in \mathbb{I} } $ and $ \left\lbrace V_{i}\bigcap K\right\rbrace _{i\in \mathbb{I} } $ constitute $\rm W.F.F $ for $ K $ with respect to weights $ \left\lbrace \nu_{i}\right\rbrace _{i\in \mathbb{I} } $ and
$ \left\lbrace \mu_{i}\right\rbrace _{i\in \mathbb{I} } $ with universal woven bounds $ \mathcal{A} $ and $ \mathcal{B} $.
\end{theorem}
\begin{proof}
Let the operators $ P_{W_{i}\cap K}= P_{W_{i}}\left( P_{K }\right) $ and $ P_{V_{i}\cap K}= P_{V_{i}}\left( P_{K}\right) $ be orthogonal projections of $ \mathcal{H} $ onto $ W_{i}\bigcap K $ and $ V_{i}\bigcap K $, respectively. Then for every $ f\in K $, we can write:
\begin{eqnarray*}
& & \sum_{i\in\sigma}\nu_{i}^{2}\Vert P_{W_{i}}\left( f\right) \Vert^{2} +\sum_{i\in\sigma^{c}}\mu_{i}^{2}\Vert P_{V_{i}}\left( f\right) \Vert^{2}\\
&= & \sum_{i\in\sigma}\nu_{i}^{2}\Vert P_{W_{i}}\left( P_{K}\left( f\right) \right)  \Vert^{2} +\sum_{i\in\sigma^{c}}\mu_{i}^{2}\Vert P_{V_{i}}\left( P_{K}\left( f\right)\right) \Vert^{2}\\
&= & \sum_{i\in\sigma}\nu_{i}^{2}\Vert P_{W_{i}\cap K }\left( f\right) \Vert^{2} +\sum_{i\in\sigma^{c}}\mu_{i}^{2}\Vert P_{V_{i}\cap K}\left( f\right) \Vert^{2},
\end{eqnarray*}
which implies the result.
\end{proof}

The next proposition shows that every weaving of fusion Bessels, automatically has upper Bessel bound.
\begin{proposition}
\label{P1}
Let $ \left\lbrace W_{ij}\right\rbrace _{i\in\mathbb{I} } $ be a fusion Bessel sequence of subspaces for $ \mathcal{H} $ with bounds $ \mathcal{B}_{j} $ for all $ j\in \left[ m\right]  $. Then every weaving of this sequence is a Bessel sequence.
\end{proposition}
\begin{proof}
For every partition $ \left\lbrace\sigma_{j} \right\rbrace _{j\in\left[ m\right] } $, such that $ \sigma_{j}\subset\mathbb{I} $ for $ j\in\left[ m\right] $ and for $ f\in\mathcal{H} $, we have
$$ \sum_{j=1}^{m}\sum_{i\in\sigma_{j}}\nu^{2}\Vert P_{W_{ij}}\left(f \right) \Vert^{2} \leq \sum_{j=1}^{m}\sum_{i=1}^{\infty}\nu^{2}\Vert P_{W_{ij}}\left(f \right) \Vert^{2}\leq\sum_{j=1}^{m}\mathcal{B}_{j}\Vert f\Vert^{2}. $$
\end{proof}
\section*{application}
Suppose $ \left\lbrace e_{k}\right\rbrace _{k=1}^{\infty }  $ be an orthonormal basis of $ \mathcal{H} $ and $ \mathcal{H}=\ell^{2}\left( \mathbb{N}\right)  $. In the next examples the indexing set $ \mathbb{I}=\mathbb{N} $ is the natural numbers set. For every $ i\in\mathbb{N} $, $ H_{i}=\overline{{\rm span}}\left\lbrace e_{k}\right\rbrace _{k=i}^{\infty } $ and $ \left\lbrace e_{ij}\right\rbrace _{j=1}^{\infty}= \left\lbrace e_{i+j-1}\right\rbrace _{j=1}^{\infty} $ is an orthonormal basis for $ H_{i} $.
\\
\begin{example}
\label{E1}
Let $ \left\lbrace P_{i}\right\rbrace _{i=1}^{\infty }  $
 and $ \left\lbrace P_{i}^{\prime}\right\rbrace _{i=1}^{\infty }  $ be the family of orthogonal projections
$ P_{i}:\mathbb{H}\longrightarrow \overline{{\rm span}}\left\lbrace e_{i}\right\rbrace  $ and $ P^{\prime}_{i}:\mathbb{H}\longrightarrow \overline{{\rm span}}\left\lbrace e_{i}, e_{i+1}\right\rbrace $ for each $ i\in\mathbb{N} $. Also let $ f_{i,j}=P_{i}(e_{i,j})$ and $ g_{i,j}=P^{\prime}_{i}(e_{i,j}) $. Then we have:
\begin{eqnarray*}
 f_{i,j}= P_{i}\left( e_{i+j-1}\right) =
\left\{\begin{array}{ll}
e_{i} \quad \quad  j=1 \\ 0 \quad\quad j > 1
\end{array}\right. \quad, \quad 
g_{i,j}= P_{i}^{\prime}\left( e_{i+j-1}\right) = \left\{\begin{array}{ll}
e_{i} \quad \quad j=1 \\
e_{i+1} \quad j=2 \\
0 \quad\quad j > 2
\end{array}\right. ,
\end{eqnarray*}
therefore
$$ \sum_{i=1}^{\infty}\sum_{j=1}^{\infty}\left\vert \left\langle f, f_{i,j}\right\rangle \right\vert^{2} =\sum_{i=1}^{\infty}\left\vert \left\langle f, f_{i,1}\right\rangle \right\vert^{2}=\sum_{i=1}^{\infty}\left\vert \left\langle f, e_{i}\right\rangle \right\vert^{2}=\Vert f\Vert^{2}. $$
So 
$ \left\lbrace f_{i,j}\right\rbrace _{i,j=1}^{\infty} $ is a tight frame with bound $ \mathcal{A}=\mathcal{B}=1 $. Also:
\begin{eqnarray*}
\sum_{i=1}^{\infty}\sum_{j=1}^{\infty}\left\vert \left\langle f, g_{i,j}\right\rangle \right\vert^{2}&=& \sum_{i=1}^{\infty}\left\vert \left\langle f, g_{i,1}\right\rangle \right\vert^{2}+\sum_{i=1}^{\infty}\left\vert \left\langle f, g_{i,2}\right\rangle \right\vert^{2}\\ &=& \sum_{i=1}^{\infty}\left\vert \left\langle f, e_{i}\right\rangle \right\vert^{2}+\sum_{i=1}^{\infty}\left\vert \left\langle f, e_{i+1}\right\rangle \right\vert^{2}\\ &=& 2\Vert f\Vert^{2} - \left\vert \left\langle f, e_{1}\right\rangle \right\vert^{2},
\end{eqnarray*}
this shows that 
$ \left\lbrace g_{i,j}\right\rbrace _{i,j=1}^{\infty} $
is a frame with bounds $ \mathcal{A}=1 $ and $ \mathcal{B}=2 $, such that these frames constitute woven frames. Because for arbitrary set $ \sigma\subset \mathbb{N} $ and for every $ f\in\mathcal{H} $, we have
\begin{eqnarray*}
\Vert f\Vert^{2}&\leq &\sum_{i\in\sigma}\sum_{j=1}^{\infty}\vert\langle f,f_{i,j}\rangle\vert^{2}+\sum_{i\in\sigma^{c}}\sum_{j=1}^{\infty}\vert\langle f, g_{i,j}\rangle\vert^{2}\\&=& \sum_{i\in\sigma}\vert \langle f, f_{i,1}\rangle\vert^{2}+ \sum_{i\in\sigma^{c}}\vert \langle f, g_{i,1} \rangle\vert^{2}+ \sum_{i\in\sigma^{c}}\vert \langle f, g_{i,2}\rangle\vert^{2}\\&=& \sum_{i\in\sigma}\vert \langle f, e_{i}\rangle\vert^{2}+ \sum_{i\in\sigma^{c}}\vert \langle f, e_{i}\rangle\vert^{2}+ \sum_{i\in\sigma^{c}}\vert \langle f, e_{i+1} \rangle\vert^{2}\\&\leq &2 \sum_{i=1}^{\infty}\vert \langle f, e_{i}\rangle\vert^{2}\\&=& \Vert f\Vert^{2}.
\end{eqnarray*}
Now, if for every $ i\in\mathbb{N} $, we assume the set $ \mathbb{J}_{i}=\mathbb{N} $, $ W_{i}=\overline{\rm span}\left\lbrace f_{i,j}\right\rbrace_{j\in\mathbb{J}_{i}} $ and $ V_{i}=\overline{\rm span}\left\lbrace g_{i,j}\right\rbrace_{j\in\mathbb{J}_{i}}$. By Theorem $\ref{T1} $ , $\left\lbrace W_{i}\right\rbrace _{i=1}^{\infty}$ and $ \left\lbrace V_{i}\right\rbrace _{i=1}^{\infty} $ constitute W.F.F with weights $ \nu_{i}=\mu_{i}=1 (\forall i\in \mathbb{N})$.
\end{example}

\begin{example}
Suppose $ \left\lbrace P_{i}\right\rbrace _{i=1}^{\infty }  $ and $ \left\lbrace P_{i}^{\prime}\right\rbrace _{i=2}^{\infty }  $ are same as in Example $ \ref{E1}$, exept $ P_{1}^{\prime } $. Then $ \left\lbrace f_{i,j}\right\rbrace _{i, j=1}^{\infty} $ and $ \left\lbrace g_{i,j}\right\rbrace _{i, j=1}^{\infty} $ don't constitute woven frames, since for $ \sigma = \mathbb{N}\setminus\left\lbrace 1\right\rbrace $, we have:
\begin{eqnarray*}
&&\sum_{i\in\sigma}\sum_{j=1}^{\infty}\vert\langle e_{1}, f_{i,j}\rangle\vert^{2}+ \sum_{i\in\sigma^{c}}\sum_{j=1}^{\infty}\vert\langle e_{1}, g_{i,j}\rangle\vert^{2}\\
&=& \sum_{i\in\sigma}\vert\langle e_{1}, P_{i}\left( e_{i,1}\right)\rangle \vert^{2}+
\sum_{i\in\sigma^{c}}\vert \langle e_{1}, P_{i}^{\prime}\left( e_{i,1}\right) \rangle\vert^{2}\\
&=& \sum_{i\in\sigma}\vert \langle e_{1}, e_{i}\rangle\vert^{2}+
 \sum_{i\in\sigma^{c}, i\neq 1}\vert \langle e_{1}, e_{i}) \rangle\vert^{2} +
 \sum_{i\in\sigma^{c}, i\neq 1}\vert \langle e_{1}, e_{i+1}\rangle \vert^{2} \\
&=&\sum_{i\in\mathbb{I}\setminus{1}}\vert\langle e_{1}, e_{i}\rangle\vert^{2}+
\vert \langle e_{1}, 0\rangle\vert^{2}\\
&=& 0 < \mathcal{A}\Vert e_{1}\Vert^{2}.
\end{eqnarray*}
This contradiction and Theorem $ \ref{T1} $ show that $\left\lbrace W_{i}\right\rbrace _{i=1}^{\infty}$ and $ \left\lbrace V_{i}\right\rbrace _{i=1}^{\infty} $ are not W.F.F .
\end{example}
Next theorem is extending Lemma 4.3 \cite{wov1}. In the following, we show that if one of the weavings does not satisfy in the lower bound condition, so the frames shall not form a W.F.F:
\begin{theorem}
\label{T2}
Suppose $ \left\lbrace W_{i}\right\rbrace _{i\in\mathbb{I} } $ and $ \left\lbrace V_{i}\right\rbrace _{i\in\mathbb{I} } $ be fusion frames for $ \mathcal{H} $ with respect to $ \left\lbrace \nu_{i}\right\rbrace _{i\in\mathbb{I} } $
and $ \left\lbrace \mu_{i}\right\rbrace _{i\in\mathbb{I} } $ and also let for every two disjoint finite sets $ I, J\subset\mathbb{I} $ and every $ \varepsilon >0 $, there exist subsets $ \sigma , \delta\subset \mathbb{I}\setminus\left( I\bigcup J\right)  $ such that the lower fusion frame bound of
$ \left\lbrace W_{i}\right\rbrace _{i\in (I\cup \sigma)}\bigcup \left\lbrace V_{i}\right\rbrace _{i\in (J\bigcup\delta)} $ is less than $ \varepsilon $. Then there exists $ \mathcal{M}\subset\mathbb{I}$ so that
$ \left\lbrace W_{i}\right\rbrace _{i\in\mathcal{M}}\bigcup\left\lbrace V_{i}\right\rbrace _{i\in\mathcal{M}^{c} } $ is not a fusion frame. Hence $ \left\lbrace W_{i}\right\rbrace _{i\in\mathbb{I} } $ and $ \left\lbrace V_{i}\right\rbrace _{i\in\mathbb{I} } $ are not ${\rm W.F.F}.$
\end{theorem}
\begin{proof}
Let $ \varepsilon >0 $ be arbitrary. By hypothesis, for $ I_{0}=J_{0}=\phi $, we can choose
$ \sigma_{1}\subset\mathbb{I} $, so that if $ \delta_{1}=\sigma_{1}^{c} $, then the lower fusion frame bound of
$ \left\lbrace W_{i}\right\rbrace _{i\in\sigma_{1}}\bigcup\left\lbrace V_{i}\right\rbrace _{i\in\sigma_{1}^{c} } $ is less than $\varepsilon$ . Thus there exists $ f_{1}\in \mathcal{H} $, with
 $ \Vert f_{1}\Vert=1 $ such that
$$\sum_{i\in\sigma_{1}}\nu_{i}^{2}\Vert P_{W_{i}}(f_{1})\Vert^{2}+
\sum_{i\in\delta_{1}}\mu_{i}^{2}\Vert P_{V_{i}}(f_{1})\Vert^{2}<\varepsilon .$$
Since $ \left\lbrace W_{i}\right\rbrace _{i\in\mathbb{I} } $ and $ \left\lbrace V_{i}\right\rbrace _{i\in\mathbb{I} } $ are fusion frames, so
$$\sum_{i=1}^{\infty}\nu_{i}^{2}\Vert P_{W_{i}}(f_{1})\Vert^{2}+\sum_{i=1}^{\infty}\mu_{i}^{2}\Vert P_{V_{i}}(f_{1})\Vert^{2}<\infty , $$
 therefor there is a positive integer $ k_{1} $ such that
 $$ \sum_{i=k_{1}+1}^{\infty}\nu_{i}^{2}\Vert P_{W_{i}}(f_{1})\Vert^{2}+\sum_{i=k_{1}+1}^{\infty}\mu_{i}^{2}\Vert P_{V_{i}}(f_{1})\Vert^{2}<\infty .$$
Let
$ I_{1}=\sigma_{1}\bigcap\left[ k_{1}\right]  $ and $ J_{1}=\delta_{1}\bigcap\left[ k_{1}\right]  $. Then
$ I_{1}\bigcap J_{1}=\phi $ and $ I_{1}\bigcup J_{1}=\left[ k_{1}\right]  $. By assumption, there are subsets
$ \sigma_{2}, \delta_{2}\subset \left[ k_{1}\right] ^{c} $ with
$ \delta_{2}=\left[ k_{1}\right] ^{c}\setminus\sigma_{2} $ such that
the lower fusion frame bound of
$ \left\lbrace W_{i}\right\rbrace _{i\in (I\cup\sigma_{2})}\bigcup\left\lbrace V_{i}\right\rbrace_{i\in(J_{2}\cup\delta_{2}) } $ is less than $\frac{\varepsilon}{2}$, so there exists a vector $ f_{2}\in\mathcal{H} $ with $ \Vert f_{2}\Vert =1 $, such that
$$\sum_{i\in (I_{1}\cup\sigma_{2})}\nu_{i}^{2}\Vert P_{W_{i}}(f_{2})\Vert^{2}+\sum_{i\in(J_{1}\cup\delta_{2})}\mu_{i}^{2}\Vert P_{V_{i}}(f_{2})\Vert^{2}<\frac{\varepsilon}{2}.$$
 Similarly, there is a $ k_{2}>k_{1} $ such that
$$ \sum_{i=k_{2}+1}^{\infty}\nu_{i}^{2}\Vert P_{W_{i}}(f_{2})\Vert^{2}+\sum_{i=k_{2}+1}^{\infty}\mu_{i}^{2}\Vert P_{V_{i}}(f_{2})\Vert^{2}<\frac{\varepsilon }{2} .$$
 Set $ I_{2}=I_{1}\bigcup\left( \sigma_{2}\bigcap\left[ k_{2}\right] \right)  $ and
 $ J_{2}=J_{1}\bigcup\left(\delta_{2}\bigcap\left[ k_{2}\right] \right)  $.
 Not that $ I_{2}\bigcap J_{2}=\phi  $ and $ I_{2}\bigcup J_{2}=\left[ k_{2}\right]  $. Thus by induction, there are:
\begin{enumerate}
\item[(i)]
 a sequence of natural numbers $ \left\lbrace k_{i}\right\rbrace _{i\in \mathbb{I} } $ with $ k_{i}<k_{i+1} $ for all $ i\in\mathbb{I} $,
\item[(ii)] a sequence of vectors $ \left\lbrace f_{i}\right\rbrace _{i\in\mathbb{I} } $ from $ \mathcal{H} $ with $ \Vert f_{i}\Vert=1 $ for all $ i\in\mathbb{I} $,
\item[(iii)]
 subsets $ \sigma_{i}\subset\left[ k_{i-1}\right] ^{c}, \delta_{i}=\left[ k_{i-1}\right] ^{c}\setminus\sigma_{i}, i\in\mathbb{I}$ and
\item[(iv)] $ I_{i}=I_{i-1}\bigcup\left( \sigma_{i}\bigcap \left[ k_{i} \right] \right) , J_{i}=J_{i-1}\bigcup\left( \delta_{i}\bigcap \left[ k_{i} \right] \right),  i\in\mathbb{I} $
which are abiding both:
\begin{eqnarray} \label{I}
\sum_{i\in (I_{n-1}\cup\sigma_{n})}\nu_{i}^{2}\Vert P_{W_{i}}(f_{n})\Vert^{2}+
 \sum_{i\in(J_{n-1}\cup\delta_{n})}\mu_{i}^{2}\Vert P_{V_{i}}(f_{n})\Vert^{2}<\dfrac{\varepsilon}{n}
 \end{eqnarray}
 and
 \begin{eqnarray}\label{II}
 \sum_{i=k_{n}+1}^{\infty}\nu_{i}^{2}\Vert P_{W_{i}}(f_{n})\Vert^{2}+
\sum_{i=k_{n}+1}^{\infty}\mu_{i}^{2}\Vert P_{V_{i}}(f_{n})\Vert^{2}<\dfrac{\varepsilon }{n}.
\end{eqnarray}

\end{enumerate}
By construction $ I_{i}\bigcap J_{i}=\left\lbrace \right\rbrace  $ and $ I_{i}\bigcup J_{i}=\left[ k_{i}\right]  $ , if we suppose that
$\mathcal{M}=\bigcup_{i=1}^{\infty}I_{i} $ then $ \mathcal{M}^{c}= \bigcup_{i=1}^{\infty}J_{i} $ such that $ \mathcal{M}\bigcup\mathcal{M}^{c}=\mathbb{I} $, then we consequence from inequalities  (\ref{I}) and (\ref{II}):
\begin{eqnarray*}
&& \sum_{i\in \mathcal{M}}\nu_{i}^{2}\Vert P_{W_{i}}(f_{i})\Vert^{2}+
\sum_{i\in\mathcal{M}^{c}}\mu_{i}^{2}\Vert P_{V_{i}}(f_{i})\Vert^{2}\\
&=& \left( \sum_{i\in I_{n}}\nu_{i}^{2}\Vert P_{W_{i}}(f_{n})\Vert^{2}+
\sum_{i\in J_{n}}\mu_{i}^{2}\Vert P_{V_{i}}(f_{n})\Vert^{2}\right)\\
&+&\left( \sum_{i\in \mathcal{M}\cap\left[ k_{n}\right] ^{c}}\nu_{i}^{2}\Vert P_{W_{i}}(f_{i})\Vert^{2}+
\sum_{i\in\mathcal{M}^{c}\cap\left[ k_{n}\right] ^{c}}\mu_{i}^{2}\Vert P_{V_{i}}(f_{i})\Vert^{2}\right) \\
&\leq &\left( \sum_{i\in I_{n-1}\cup\sigma_{n}}\nu_{i}^{2}\Vert P_{W_{i}}(f_{n})\Vert^{2}+
\sum_{i\in J_{n-1}\cup\delta_{n}}\mu_{i}^{2}\Vert P_{V_{i}}(f_{n})\Vert^{2}\right)\\
&+&\left( \sum_{i=k_{n}+1}^{\infty}\nu_{i}^{2}\Vert P_{W_{i}}(f_{i})\Vert^{2}+
\sum_{i=k_{n}+1}^{\infty}\mu_{i}^{2}\Vert P_{V_{i}}(f_{i})\Vert^{2}\right) \\
&< &\frac{\varepsilon}{n}+\frac{\varepsilon }{n}=\frac{2\varepsilon }{n}.
\end{eqnarray*}
Therefor the lower fusion frame of $\left\lbrace W_{i}\right\rbrace _{i\in\mathcal{M}}\bigcup\left\lbrace V_{i}\right\rbrace _{i\in\mathcal{M}^{c} } $ is zero, that is contradiction. Thus $ \left\lbrace W_{i}\right\rbrace _{i\in\mathbb{I} } $ and
$ \left\lbrace V_{i}\right\rbrace _{i\in\mathbb{I} } $ can not be W.F.F .
\end{proof}
This section is concluded by showing that the upper bound in Proposition \ref{P1} can not be optimal for W.F.F.
\begin{proposition}
Suppose that $ \left\lbrace W_{i}\right\rbrace _{i\in\mathbb{I} } $ and $ \left\lbrace V_{i}\right\rbrace _{i\in\mathbb{I} } $ be fusion frames for $ \mathcal{H} $ with respect to weights $ \left\lbrace \nu_{i}\right\rbrace _{i\in\mathbb{I} } $ and
$ \left\lbrace \mu_{i}\right\rbrace _{i\in\mathbb{I} } $ and also with optimal upper fusion frame bounds $ \mathcal{B}_{1} $ and $ \mathcal{B}_{2} $ such that constitute $\rm W.F.F$. Then $ \mathcal{B}_{1}+\mathcal{B}_{2} $ can not be the optimal upper woven bound.
\end{proposition}
\begin{proof}
By contradiction, we assume that $ \mathcal{B}_{1}+\mathcal{B}_{2} $ is the smallest upper weaving bound for all possible weavings. Then by definition of optimal upper bound, we can choose $ \sigma\subset I $ and $ \Vert f\Vert=1 $, such that
\begin{eqnarray*}
\sup_{\Vert f\Vert=1}\left( \sum_{i\in\sigma}\nu_{i}^{2}\Vert P_{W_{i}}(f)\Vert^{2}+
\sum_{i\in\sigma^{c}}\mu_{i}^{2}\Vert P_{V_{i}}(f)\Vert^{2}\right) =\mathcal{B}_{1}+\mathcal{B}_{2}.
\end{eqnarray*}
Using of supreme property, for every $ \varepsilon>0 $, there exist $ f\in\mathcal{H} $, such that
\begin{eqnarray*}
\sum_{i\in \sigma}\nu_{i}^{2}\Vert P_{W_{i}}(f)\Vert^{2}+
\sum_{i\in\sigma^{c}}\mu_{i}^{2}\Vert P_{V_{i}}(f)\Vert^{2} \geq \mathcal{B}_{1}+\mathcal{B}_{2}-\varepsilon ,
\end{eqnarray*}
and using of upper fusion frame property, we have
\begin{eqnarray*}
\sum_{i\in \mathbb{I}}\nu_{i}^{2}\Vert P_{W_{i}}(f)\Vert^{2}+
\sum_{i\in\mathbb{I}}\mu_{i}^{2}\Vert P_{V_{i}}(f)\Vert^{2} \leq \mathcal{B}_{1}+\mathcal{B}_{2}.
\end{eqnarray*}
So:
\begin{eqnarray*}
\sum_{i\in \mathbb{I}\setminus\sigma}\nu_{i}^{2}\Vert P_{W_{i}}(f)\Vert^{2}+
\sum_{i\in\mathbb{I}\setminus\sigma^{c}}\mu_{i}^{2}\Vert P_{V_{i}}(f)\Vert^{2} \leq\varepsilon.
\end{eqnarray*}
Now, if we assume $ \sigma_{1}=\mathbb{I}\setminus\sigma $, then $ \sigma_{1}^{c}=\mathbb{I}\setminus\sigma^{c} $. Therefor
\begin{eqnarray*}
\sum_{i\in\sigma_{1}}\nu_{i}^{2}\Vert P_{W_{i}}(f)\Vert^{2}+
\sum_{i\in\sigma_{1}^{c}}\mu_{i}^{2}\Vert P_{V_{i}}(f)\Vert^{2} \leq\varepsilon ,
\end{eqnarray*}
and this shows that there is a weaving for which the lower frame bound approachs zero. Theorem $ \ref{T2}$ gives that $ \left\lbrace W_{i}\right\rbrace _{i\in\mathbb{I} } $ and $ \left\lbrace V_{i}\right\rbrace _{i\in\mathbb{I} } $ are not W.F.F, which is contradiction.
\end{proof}
\begin{proposition}
Let $ \left\lbrace W_{i}\right\rbrace _{i\in J} $ and $ \left\lbrace V_{i}\right\rbrace _{i\in J} $ be $\rm W.F.F$, with respect to weights $ \left\lbrace \nu_{i}\right\rbrace _{i\in J} $ and $ \left\lbrace \mu_{i}\right\rbrace _{i\in J} $ such that
$ J\subset\mathbb{I} $. Then $ \left\lbrace W_{i}\right\rbrace _{i\in\mathbb{I}} $ and $ \left\lbrace V_{i}\right\rbrace _{i\in\mathbb{I}} $ are $\rm W.F.F$,  with weights
$ \left\lbrace \nu_{i}\right\rbrace _{i\in\mathbb{I}} $ and
$ \left\lbrace \mu_{i}\right\rbrace _{i\in\mathbb{I}} $.
\end{proposition}
\begin{proof}
Let the positive constant $ \mathcal{A} $ be the lower woven bound for $ \left\lbrace W_{i}\right\rbrace _{i\in J} $ and
$ \left\lbrace V_{i}\right\rbrace _{i\in J} $. Then for every $ \sigma\subset \mathbb{I} $ and $ f\in\mathcal{H} $, we have
\begin{eqnarray*}
\mathcal{A}\Vert f\Vert^{2}&\leq & \sum_{i\in\sigma\cap J}\nu_{i}^{2}\Vert P_{W_{i}}(f)\Vert^{2} +
\sum_{i\in\sigma^{c}\cap J}\mu_{i}^{2}\Vert P_{V_{i}}(f)\Vert^{2}\\
&\leq & \sum_{i\in\sigma}\nu_{i}^{2}\Vert P_{W_{i}}(f)\Vert^{2} +
\sum_{i\in\sigma^{c}}\mu_{i}^{2}\Vert P_{V_{i}}(f)\Vert^{2}\\
&\leq &\left( \mathcal{B}_{W}+\mathcal{B}_{V}\right) \Vert f\Vert^{2},
\end{eqnarray*}
where $ \mathcal{B}_{W} $ and $ \mathcal{B}_{V} $ are upper fusion frame bounds for $ \left\lbrace W_{i}\right\rbrace _{i\in J} $ and $ \left\lbrace V_{i}\right\rbrace _{i\in J} $, respectively.
\end{proof}
\begin{proposition}
Suppose $ \left\lbrace W_{i}\right\rbrace _{i\in\mathbb{I} } $ and $ \left\lbrace V_{i}\right\rbrace _{i\in\mathbb{I} } $ are fusion frames for $ \mathcal{H} $ with respect to weights $ \left\lbrace \nu_{i}\right\rbrace _{i\in\mathbb{I} } $ with universal woven bounds $ \mathcal{A} $ and $ \mathcal{B} $. For some constant $ 0<\mathcal{D}<\mathcal{A} $ and $ J\subset\mathbb{I} $, if we have:
$$ \sum_{i\in J}\nu_{i}^{2}\Vert P_{W_{i}}(f)\Vert^{2}\leq\mathcal{D}\Vert f\Vert^{2},\quad\forall f\in \mathcal{H}. $$
Then $ \left\lbrace W_{i}\right\rbrace _{i\in\mathbb{I}\setminus J } $ and
$ \left\lbrace V_{i}\right\rbrace _{i\in\mathbb{I}\setminus J } $ are fusion frames for $ \mathcal{H} $ and are
$\rm W.F.F$ with universal lower and upper woven bounds $ \mathcal{A}-\mathcal{D} $ and $ \mathcal{B} $, respectively.
\end{proposition}
\begin{proof}
Assume $ \sigma\subset \mathbb{I}\setminus J $. Then for all $ f\in\mathcal{H} $
\begin{eqnarray*}
&&\sum_{i\in\sigma}\nu_{i}^{2}\Vert P_{W_{i}}(f)\Vert^{2} +
\sum_{i\in (\mathbb{I}\setminus J)\setminus \sigma}\mu_{i}^{2}\Vert P_{V_{i}}(f)\Vert^{2}\\
&=&
\left( \sum_{i\in\sigma\cup J}\nu_{i}^{2}\Vert P_{W_{i}}(f)\Vert^{2}-\sum_{i\in J}\nu_{i}^{2}\Vert P_{W_{i}}(f)\Vert^{2}\right) +\sum_{i\in (\mathbb{I}\setminus J)\setminus \sigma}\mu_{i}^{2}\Vert P_{V_{i}}(f)\Vert^{2}\\
&=& \left(\sum_{i\in\sigma\cup J}\nu_{i}^{2}\Vert P_{W_{i}}(f)\Vert^{2} +\sum_{i\in\mathbb{I}\setminus(J\cup\sigma)}\mu_{i}^{2}\Vert P_{V_{i}}(f)\Vert^{2}\right)
-\sum_{i\in J}\nu_{i}^{2}\Vert P_{W_{i}}(f)\Vert^{2}\\
&\geq &\left( \mathcal{A}-\mathcal{D}\right) \Vert f\Vert^{2}.
\end{eqnarray*}
 For upper woven bound, we have
\begin{eqnarray*}
&&\sum_{i\in\sigma}\nu_{i}^{2}\Vert P_{W_{i}}(f)\Vert^{2} +
\sum_{i\in (\mathbb{I}\setminus J)\setminus \sigma}\mu_{i}^{2}\Vert P_{V_{i}}(f)\Vert^{2}\\
&\leq & \sum_{i\in\sigma \cup J}\nu_{i}^{2}\Vert P_{W_{i}}(f)\Vert^{2} +
\sum_{i\in\mathbb{I}\setminus (\sigma\cup J)}\mu_{i}^{2}\Vert P_{V_{i}}(f)\Vert^{2}\\
&\leq & \mathcal{B}\Vert f\Vert^{2}.
\end{eqnarray*}
Thus $ \left\lbrace W_{i}\right\rbrace _{i\in\mathbb{I}\setminus J } $ and
$ \left\lbrace V_{i}\right\rbrace _{i\in\mathbb{I}\setminus J } $ are W.F.F.
Now, if we take $ \sigma =\mathbb{I} $ and $ \sigma^{c}=\phi $, then $ \left\lbrace W_{i}\right\rbrace _{i\in\mathbb{I}\setminus J} $ is fusion frame:
\begin{eqnarray*}
&&\sum_{i\in\mathbb{I}\setminus J }\nu_{i}^{2}\Vert P_{W_{i}}(f)\Vert^{2}\\
&=&\sum_{i\in\mathbb{I}}\nu_{i}^{2}\Vert P_{W_{i}}(f)\Vert^{2}-\sum_{i\in J }\nu_{i}^{2}\Vert P_{W_{i}}(f)\Vert^{2}\\
&=& \sum_{i\in\sigma}\nu_{i}^{2}\Vert P_{W_{i}}(f)\Vert^{2}+\sum_{i\in\sigma^{c}}\mu_{i}^{2}\Vert P_{V_{i}}(f)\Vert^{2}
-\sum_{i\in J }\nu_{i}^{2}\Vert P_{W_{i}}(f)\Vert^{2}\\
&\geq & \left(\mathcal{A}-\mathcal{D} \right) \Vert f\Vert^{2}.
\end{eqnarray*}
Similar to above, we can demonstrate that $ \left\lbrace V_{i}\right\rbrace_{i\in\mathbb{I}\setminus J}$ is a fusion frame with same bounds.
\end{proof}
\section{Perturbation of woven of subspaces (W.F.F)}
It is well known that the perturbation theory is a paramount component in the study of frames. In this section, we show that those of fusion frames that are small perturbations of each other, constitute W.F.F. We start this section with Paley-Wiener perturbation of weaving fusion frames and continue two results of perturbations in the sequel.
\begin{theorem}
\label{T4.1}
Let $ \left\lbrace W_{i}\right\rbrace _{i\in\mathbb{I} } $ and $ \left\lbrace V_{i}\right\rbrace _{i\in\mathbb{I} } $ be fusion frames for $ \mathcal{H} $ with weights $ \left\lbrace \nu_{i}\right\rbrace _{i\in\mathbb{I} } $ and
$ \left\lbrace \mu_{i}\right\rbrace _{i\in\mathbb{I}} $ and fusion frame bounds $ \left( \mathcal{A}_{W}, \mathcal{B}_{W}\right)  $ and $ \left( \mathcal{A}_{V}, \mathcal{B}_{V}\right)  $, respectively. If there exist constants $ 0<\lambda_{1}, \lambda_{2},\mu<1 $ such that:
$$ \dfrac{2}{\mathcal{A}_{W} }\left( \sqrt{\mathcal{B}_{W}}+\sqrt{\mathcal{B}_{V}}\right)
 \left( \lambda_{1}\sqrt{\mathcal{B}_{W}}+\lambda_{2}\sqrt{\mathcal{B}_{V}}+\mu\right)
\leq 1 $$
and
\begin{eqnarray}
\label{3}
\Vert T_{W, \nu}(f)-T_{V,\mu}(f)\Vert\leq \lambda_{1}\Vert T_{W, \nu}(f)\Vert+\lambda_{2}\Vert T_{V,\mu}(f)\Vert +\mu .
\end{eqnarray}
Then $ \left\lbrace W_{i}\right\rbrace _{i\in\mathbb{I} } $ and $ \left\lbrace V_{i}\right\rbrace _{i\in\mathbb{I} } $ are $\rm W.F.F $ .
\end{theorem}
\begin{proof}
For each $ \sigma\subset\mathbb{I} $, we define the bounded operators
$$ T_{W, \nu}^{\sigma}:\left( \sum_{i\in\sigma}\bigoplus W_{i}\right) _{\ell^{2}}\longrightarrow\mathcal{H}, \quad
T_{W, \nu}^{\sigma}(f)=\sum_{i\in\sigma}\nu_{i}f_{i}$$
and
$$ T_{V, \mu}^{\sigma}:\left( \sum_{i\in\sigma}\bigoplus V_{i}\right) _{\ell^{2}}\longrightarrow\mathcal{H},\quad T_{V, \mu}^{\sigma}(f)=\sum_{i\in\sigma}\mu_{i}f_{i}~. $$
For every
$ f=\left\lbrace f_{i}\right\rbrace _{i\in\mathbb{I} }\in\left( \sum_{i\in\sigma}\bigoplus W_{i}\right) _{\ell^{2}}$,
note that $ \Vert T_{W, \nu}^{\sigma}(f)\Vert\leq\Vert T_{W, \nu}(f)\Vert $,
$ \Vert T_{V, \mu}^{\sigma}(f)\Vert\leq\Vert T_{V, \mu}(f)\Vert $ and
$ \Vert T_{W, \nu}^{\sigma}(f)-T_{V, \mu}^{\sigma}(f)\Vert\leq\Vert T_{W, \nu}(f)-T_{V, \mu}(f)\Vert $. Using the statement $ (\ref{3} ) $, for every $ f\in\mathcal{H} $ and $ \sigma\subset\mathbb{I} $, we have
\begin{eqnarray*}
&&\Vert T_{W,\nu}^{\sigma} U_{W,\nu}^{\sigma}(f)-T_{V,\mu}^{\sigma} U_{V,\mu}^{\sigma}(f)\Vert \\
&=& \Vert T_{W,\nu}^{\sigma} U_{W,\nu}^{\sigma}(f)- T_{W,\nu}^{\sigma}U_{V,\mu}^{\sigma}(f)+
T_{W,\nu}^{\sigma}U_{V,\mu}^{\sigma}(f)-T_{V,\mu}^{\sigma} U_{V,\mu}^{\sigma}(f)\Vert  \\
&\leq &\Vert T_{W,\nu}^{\sigma} \left( U_{W,\nu}^{\sigma}- U_{V,\mu}^{\sigma}\right) (f)\Vert +
\Vert \left( T_{W,\nu}^{\sigma}-T_{V,\mu}^{\sigma}\right) U_{V,\mu}^{\sigma}(f)\Vert  \\
&\leq & \Vert T_{W,\nu}^{\sigma} \Vert \Vert U_{W,\nu}^{\sigma}- U_{V,\mu}^{\sigma}\Vert\Vert f\Vert +
\Vert T_{W,\nu}^{\sigma}-T_{V,\mu}^{\sigma}\Vert\Vert U_{V,\mu}^{\sigma}\Vert\Vert f\Vert  \\
&\leq & \Vert T_{W,\nu} \Vert \Vert T_{W,\nu}- T_{V,\mu}\Vert\Vert f\Vert +
\Vert T_{W,\nu}-T_{V,\mu}\Vert\Vert T_{V,\mu}\Vert\Vert f\Vert  \\
 &\leq & \Vert T_{W,\nu}-T_{V,\mu}\Vert\left( \Vert T_{W,\nu} \Vert +\Vert T_{V,\mu}\Vert \right) \Vert f\Vert \\
&\leq & \left( \lambda_{1}\sqrt{\mathcal{B}_{W}}+\lambda_{2}\sqrt{\mathcal{B}_{V}}+\mu\right)
\left( \sqrt{\mathcal{B}_{W}}+\sqrt{\mathcal{B}_{V}}\right) \Vert f\Vert\\
&\leq &\dfrac{\mathcal{A}_{W}}{2}\Vert f\Vert .
\end{eqnarray*}
Now by using above calculation, we have
\begin{eqnarray*}
&& \Vert T_{W,\nu}^{\sigma^{c}} U_{W,\nu}^{\sigma^{c}}(f)+T_{V,\mu}^{\sigma} U_{V,\mu}^{\sigma}(f)\Vert \\
&=&\Vert T_{W,\nu}^{\sigma^{c}} U_{W,\nu}^{\sigma^{c}}(f)+T_{W,\nu}^{\sigma} U_{W,\nu}^{\sigma}(f)-
 T_{W,\nu}^{\sigma} U_{W,\nu}^{\sigma}(f)-T_{V,\mu}^{\sigma} U_{V,\mu}^{\sigma}(f)\Vert \\
&=& \Vert T_{W,\nu} U_{W,\nu}(f)-T_{W,\nu}^{\sigma} U_{W,\nu}^{\sigma}(f)-T_{V,\mu}^{\sigma} U_{V,\mu}^{\sigma}(f)\Vert \\
&\geq &\Vert T_{W,\nu} U_{W,\nu}(f)\Vert -
\Vert T_{V,\mu}^{\sigma} U_{V,\mu}^{\sigma}(f)- T_{W,\nu}^{\sigma} U_{W,\nu}^{\sigma}(f)\Vert \\
&\geq & \mathcal{A}_{W}\Vert f\Vert -\dfrac{\mathcal{A}_{W}}{2}\Vert f\Vert\\ &=&\dfrac{\mathcal{A}_{W}}{2}\Vert f\Vert,\quad \forall f\in\mathcal{H}.
\end{eqnarray*}
This shows that $ \dfrac{\mathcal{A}_{W}}{2} $ is the universal lower woven bound. Finally, for universal upper bound, we have
\begin{eqnarray*}
\Vert T_{W,\nu}^{\sigma^{c}} U_{W,\nu}^{\sigma^{c}}(f)+T_{V,\mu}^{\sigma} U_{V,\mu}^{\sigma}(f)\Vert
&\leq & \Vert T_{W,\nu}^{\sigma^{c}} U_{W,\nu}^{\sigma^{c}}(f)\Vert +\Vert T_{V,\mu}^{\sigma} U_{V,\mu}^{\sigma}(f)\Vert \\
&\leq &\Vert T_{W,\nu} U_{W,\nu}(f)\Vert +\Vert T_{V,\mu}U_{V,\mu}(f)\Vert\\
&\leq &\left(\mathcal{B}_{W}+\mathcal{B}_{V} \right) \Vert f\Vert .
\end{eqnarray*}
\end{proof}

\begin{theorem}
Let $ \left\lbrace W_{i}\right\rbrace _{i\in\mathbb{I} } $ and $ \left\lbrace V_{i}\right\rbrace _{i\in\mathbb{I} } $ be fusion frames for $ \mathcal{H} $ with weights $ \left\lbrace \nu_{i}\right\rbrace _{i\in\mathbb{I} } $ and
$ \left\lbrace \mu_{i}\right\rbrace _{i\in\mathbb{I}} $ and fusion frame bounds $ \left( \mathcal{A}_{W}, \mathcal{B}_{W}\right)  $ and $ \left( \mathcal{A}_{V}, \mathcal{B}_{V}\right)  $, respectively and the operators $ \left( T_{W,\nu}, U_{W,\nu}\right)  $ and
$ \left( T_{V,\mu}, U_{V,\mu}\right)  $ are the synthesis and analysis operators for these frames. If there exist constants
$ 0<\lambda,\mu,\gamma <1 $, such that
$ \lambda\mathcal{B}_{W}+\mu\mathcal{B}_{\mu}+\gamma\sqrt{B_{W}}<\mathcal{A}_{W} $ and for $ f\in\mathcal{H} $ and arbitrary $ \sigma\subset\mathbb{I} $, we have
\begin{eqnarray*}
&&\Vert T_{W,\nu}^{\sigma} U_{W,\nu}^{\sigma}(f)-T_{V,\mu}^{\sigma} U_{V,\mu}^{\sigma}(f)\Vert\\ &\leq &\lambda\Vert T_{W,\nu}^{\sigma} U_{W,\nu}^{\sigma}(f)\Vert +\mu\Vert T_{V,\mu}^{\sigma} U_{V\mu}^{\sigma}(f)\Vert +\gamma \Vert U_{W,\nu}^{\sigma}(f)\Vert .
\end{eqnarray*}
Such that $ T_{W,\nu}^{\sigma} $ , $ U^{\sigma}_{W,\nu} $, $ T_{V,\mu}^{\sigma} $ and $ U_{V,\mu}^{\sigma} $ are the same as Theorem $\ref{T4.1}$. Then $ \left\lbrace W_{i}\right\rbrace _{i\in\mathbb{I} } $ and $ \left\lbrace V_{i}\right\rbrace _{i\in\mathbb{I} } $
are $\rm W.F.F $, with universal woven bounds
$$\left( \mathcal{A}_{W}-\lambda\mathcal{B}_{W}-\mu\mathcal{B}_{V}-\gamma\sqrt{B_{W}}\right) \quad, \quad\left( \mathcal{A}_{W}+\lambda\mathcal{B}_{W}+\mu\mathcal{B}_{V}+\gamma\sqrt{B_{W}}\right) . $$
\end{theorem}
\begin{proof}
By using
$ \Vert T_{W,\nu}^{\sigma} U_{W,\nu}^{\sigma}(f)\Vert\leq\mathcal{B}_{W}\Vert f\Vert $ and
$ \Vert T_{V,\mu}^{\sigma} U_{V,\mu}^{\sigma}(f)\Vert\leq\mathcal{B}_{V}\Vert f\Vert $, for any
$ \sigma\subset\mathbb{I} $ and $ f\in\mathcal{H} $, we compute
\begin{eqnarray*}
&&\Vert T_{W,\nu}^{\sigma} U_{W,\nu}^{\sigma}(f)+T_{V,\mu}^{\sigma^{c}} U_{V,\mu}^{\sigma^{c}}(f)\Vert\\
&= &\Vert T_{W,\nu}^{\sigma} U_{W,\nu}^{\sigma}(f)+T_{W,\nu}^{\sigma^{c}} U_{W,\nu}^{\sigma^{c}}(f)-
T_{W,\nu}^{\sigma^{c}} U_{W,\nu}^{\sigma^{c}}(f)+T_{V,\mu}^{\sigma^{c}} U_{V,\mu}^{\sigma^{c}}(f)\Vert\\
&=& \Vert T_{W,\nu} U_{W,\nu}(f)+T_{V,\mu}^{\sigma^{c}} U_{V,\mu}^{\sigma^{c}}(f)-
T_{W,\nu}^{\sigma^{c}} U_{W,\nu}^{\sigma^{c}}(f)\Vert\\
&\geq & \Vert T_{W,\nu} U_{W,\nu}(f)\Vert - \Vert T_{V,\mu}^{\sigma^{c}} U_{V,\mu}^{\sigma^{c}}(f)-
T_{W,\nu}^{\sigma^{c}} U_{W,\nu}^{\sigma^{c}}(f)\Vert\\
&\geq &\mathcal{A}_{W}\Vert f\Vert -\lambda\Vert T_{W,\nu}^{\sigma} U_{W,\nu}^{\sigma}(f)\Vert -
\mu\Vert T_{V,\mu}^{\sigma} U_{V,\mu}^{\sigma}(f)\Vert -\gamma \Vert U_{W,\nu}^{\sigma}(f)\Vert\\
&\geq & \left( \mathcal{A}_{W}-\lambda\mathcal{B}_{W}-\mu\mathcal{B}_{V}-\gamma\sqrt{B_{W}}\right) \Vert f\Vert.
\end{eqnarray*}
Also, for upper frame bound, we have by hypothesis and first equality of above calculation
\begin{eqnarray*}
&&\Vert T_{W,\nu}^{\sigma} U_{W,\nu}^{\sigma}(f)+T_{V,\mu}^{\sigma^{c}} U_{V,\mu}^{\sigma^{c}}(f)\Vert\\
&=& \Vert T_{W,\nu} U_{W,\nu}(f)+T_{V,\mu}^{\sigma^{c}} U_{V,\mu}^{\sigma^{c}}(f)-
T_{W,\nu}^{\sigma^{c}} U_{W,\nu}^{\sigma^{c}}(f)\Vert\\
&\leq &\Vert T_{W,\nu} U_{W,\nu}(f)\Vert + \Vert T_{V,\mu}^{\sigma^{c}} U_{V,\mu}^{\sigma^{c}}(f)-
T_{W,\nu}^{\sigma^{c}} U_{W,\nu}^{\sigma^{c}}(f)\Vert\\
&\leq &\mathcal{B}_{W}\Vert f\Vert +\lambda\Vert T_{W,\nu}^{\sigma} U_{W,\nu}^{\sigma}(f)\Vert +
\mu\Vert T_{V,\mu}^{\sigma} U_{V,\mu}^{\sigma}(f)\Vert + \gamma \Vert U_{W,\nu}^{\sigma}(f)\Vert\\
&\leq &\left( \mathcal{A}_{W}+\lambda\mathcal{B}_{W}+\mu\mathcal{B}_{V}+\gamma\sqrt{B_{W}}\right)\Vert f\Vert^{2}.
\end{eqnarray*}
Therefore fusion frames $ \left\lbrace W_{i}\right\rbrace _{i\in\mathbb{I} } $ and $ \left\lbrace V_{i}\right\rbrace _{i\in\mathbb{I} } $ are W.F.F, with aforementioned bounds.
\end{proof}

\begin{theorem}
Let $ \left\lbrace W_{i}\right\rbrace _{i\in\mathbb{I} } $ and $ \left\lbrace V_{i}\right\rbrace _{i\in\mathbb{I} } $ be fusion frames for $ \mathcal{H} $ with weights $ \left\lbrace \nu_{i}\right\rbrace _{i\in\mathbb{I}} $  and fusion frame bounds
$ \left( \mathcal{A}_{W}, \mathcal{B}_{W}\right)  $ and $ \left( \mathcal{A}_{V}, \mathcal{B}_{V}\right)  $, respectively. Also, if there exist a constant $ \mathcal{K}>0 $, such that for every
$ \sigma\subset\mathbb{I} $:
$$ \sum_{i\in\sigma}\nu_{i}^{2}\Vert P_{W_{i}}(f)-P_{V_{i}}(f)\Vert \leq\mathcal{K}\min\left\lbrace \sum_{i\in\sigma}\nu_{i}^{2}\Vert P_{W_{i}}(f)\Vert, \sum_{i\in\sigma}\nu_{i}^{2}\Vert P_{V_{i}}(f)\Vert \right\rbrace , $$
then
$ \left\lbrace W_{i}\right\rbrace _{i\in\mathbb{I} } $ and $ \left\lbrace V_{i}\right\rbrace _{i\in\mathbb{I} } $ are $\rm W.F.F $.
\end{theorem}
\begin{proof}
Let $ \sigma\subset\mathbb{I} $ be an arbitrary set. By hypothesis for every $ f\in\mathcal{H} $, we have
\begin{eqnarray*}
&& \left( \mathcal{A}_{W}+\mathcal{A}_{V}\right)\Vert f\Vert^{2} \\
&\leq &\sum_{i\in\mathbb{I}}\nu_{i}^{2}\Vert P_{W_{i}}(f)\Vert^{2}+\sum_{i\in\mathbb{I}}\nu_{i}^{2}\Vert P_{V_{i}}(f)\Vert^{2}\\
&=&\left( \sum_{i\in\sigma}\nu_{i}^{2}\Vert P_{W_{i}}(f)\Vert^{2}
+\sum_{i\in\sigma^{c}}\nu_{i}^{2}\Vert P_{W_{i}}(f)- P_{V_{i}}(f)+P_{V_{i}}(f)\Vert^{2} \right) \\
&+&\left(  \sum_{i\in\sigma}\nu_{i}^{2}\Vert P_{V_{i}}(f)-P_{W_{i}}(f)+P_{W_{i}}(f)\Vert^{2}
+\sum_{i\in\sigma}\nu_{i}^{2}\Vert P_{V_{i}}(f)\Vert^{2}\right) \\
&\leq & \sum_{i\in\sigma}\nu_{i}^{2}\Vert P_{W_{i}}(f)\Vert^{2}+
2\sum_{i\in\sigma^{c}}\nu_{i}^{2}\Vert P_{W_{i}}(f)- P_{V_{i}}(f)\Vert^{2}+
2\sum_{i\in\sigma^{c}}\nu_{i}^{2}\Vert P_{V_{i}}(f)\Vert^{2}\\
&+&2\sum_{i\in\sigma}\nu_{i}^{2}\Vert P_{W_{i}}(f)- P_{V_{i}}(f)\Vert^{2}
+2\sum_{i\in\sigma}\nu_{i}^{2}\Vert P_{W_{i}}(f)\Vert^{2}+\sum_{i\in\sigma^{c}}\nu_{i}^{2}\Vert P_{V_{i}}(f)\Vert^{2}\\
&\leq & \sum_{i\in\sigma}\nu_{i}^{2}\Vert P_{W_{i}}(f)\Vert^{2}+
2\left(\mathcal{K}\sum_{i\in\sigma^{c}}\nu_{i}^{2}\Vert P_{V_{i}}(f)+\sum_{i\in\sigma^{c}}\nu_{i}^{2}\Vert P_{V_{i}}(f)\Vert^{2}\right)\\
&+&2\left(\mathcal{K}\sum_{i\in\sigma}\nu_{i}^{2}\Vert P_{W_{i}}(f)+\sum_{i\in\sigma}\nu_{i}^{2}\Vert P_{W_{i}}(f)\Vert^{2}\right)+\sum_{i\in\sigma^{c}}\nu_{i}^{2}\Vert P_{V_{i}}(f)\Vert^{2}\\
&=&\left( 2\mathcal{K}+1\right) \left(\sum_{i\in\sigma}\nu_{i}^{2}\Vert P_{W_{i}}(f)\Vert^{2} +\sum_{i\in\sigma^{c}}\nu_{i}^{2}\Vert P_{V_{i}}(f)\Vert^{2}\right),
\end{eqnarray*}
then
$$ \dfrac{\mathcal{A}_{W}+\mathcal{A}_{V}}{2\mathcal{K}+1}\Vert f\Vert^{2}\leq\sum_{i\in\sigma}\nu_{i}^{2}\Vert P_{W_{i}}(f)\Vert^{2} +\sum_{i\in\sigma^{c}}\nu_{i}^{2}\Vert P_{V_{i}}(f)\Vert^{2}\leq\left( \mathcal{B}_{W}+\mathcal{B}_{V}\right) \Vert f\Vert^{2}.$$
Since $ \sigma\subset\mathbb{I} $ is arbitrary, therefor $ \left\lbrace W_{i}\right\rbrace _{i\in\mathbb{I} } $ and
$ \left\lbrace V_{i}\right\rbrace _{i\in\mathbb{I} } $ are W.F.F.
\end{proof}
\section{Riesz decomposition of woven of subspaces}
\begin{definition}
Woven fusion frames $ \left\lbrace W_{i}\right\rbrace _{i\in\mathbb{I}} $ and $ \left\lbrace V_{i}\right\rbrace _{i\in\mathbb{I}} $ of $ \mathcal{H} $ with respect to some weights, are called woven Riesz decomposition for $ \mathcal{H} $, if for every $ f\in\mathcal{H} $,
there are unique elements $ f_{i}\in W_{i} $ and $ g_{i}\in V_{i} $, such that $ f=\sum_{i\in\sigma}f_{i}+\sum_{i\in\sigma^{c}}g_{i} $, for any partition $ \sigma\subset\mathbb{I} $.
\end{definition}
\begin{corollary}
\label{C1}
\cite{caku1}
If $ \left\lbrace W_{i}\right\rbrace _{i\in\mathbb{I}} $ is an orthonormal basis of subspaces for $ \mathcal{H} $, Then it is also a Riesz decomposition of $ \mathcal{H} $.
\end{corollary}
\begin{theorem}
Let $ \left\lbrace W_{i}\right\rbrace _{i\in\mathbb{I}} $ and $ \left\lbrace V_{i}\right\rbrace _{i\in\mathbb{I}} $ be orthonormal fusion basis for $ \mathcal{H} $ and also $ \left\lbrace e_{ij}\right\rbrace _{j\in\mathbb{J}_{i} } $ is an orthonormal basis for $ W_{i} $ and $ V_{i} $ for every $ i\in\mathbb{I} $. Then $ \left\lbrace W_{i}\right\rbrace _{i\in\mathbb{I}} $ and
$ \left\lbrace V_{i}\right\rbrace _{i\in\mathbb{I}} $ are woven Riesz decomposition.
\end{theorem}
\begin{proof}
The hypothesis and Corollary \ref{C1}, show that $ \left\lbrace W_{i}\right\rbrace _{i\in\mathbb{I}} $ and
$ \left\lbrace V_{i}\right\rbrace _{i\in\mathbb{I}} $ are Riesz decompositions and fusion frames. Thus for every $ f\in\mathcal{H} $, we have:
$$ \mathcal{H}=\bigoplus_{i\in\mathbb{I}}W_{i}\quad \Rightarrow\quad \forall f\in\mathcal{H},~ \exists ! f_{i}\in W_{i}\quad s.t\quad f=\sum_{i\in\mathbb{I}}f_{i} $$
and
$$ \mathcal{H}=\bigoplus_{i\in\mathbb{I}}V_{i}\quad \Rightarrow\quad \forall f\in\mathcal{H},~ \exists ! f^{\prime}_{i}\in V_{i}\quad s.t\quad f=\sum_{i\in\mathbb{I}}f^{\prime}_{i}. $$
Since $ \left\lbrace e_{ij}\right\rbrace _{j\in\mathbb{J}_{i} } $ is an orthonormal basis for $ W_{i} $ and $ V_{i} $, for every
$ i\in\mathbb{I} $, therefor $ f_{i}=f^{\prime}_{i} $ and so $ W_{i} $ coincide to $ V_{i} $ for every
$ i\in\mathbb{I} $. Then for every
$ \sigma\subset\mathbb{I} $, $ \left\lbrace W_{i}\right\rbrace _{i\in\sigma} \bigcup\left\lbrace V_{i}\right\rbrace _{i\in\sigma^{c}} $ is a fusion frame and so $ \left\lbrace W_{i}\right\rbrace _{i\in\mathbb{I}} $ and
$ \left\lbrace V_{i}\right\rbrace _{i\in\mathbb{I}} $ are W.F.F .

Now from the fact that
$ \left\lbrace W_{i}\right\rbrace _{i\in\sigma} \bigcup\left\lbrace V_{i}\right\rbrace _{i\in\sigma^{c}}= \left\lbrace W_{i}\right\rbrace _{i\in\mathbb{I}}$ is orthonormal basis for $ \sigma\subset\mathbb{I} $, we conclude that
$ \left\lbrace W_{i}\right\rbrace _{i\in\sigma} \bigcup\left\lbrace V_{i}\right\rbrace _{i\in\sigma^{c}} $ is Riesz decomposition.
Since $ \sigma\subset\mathbb{I} $ is arbitrary, then $ \left\lbrace W_{i}\right\rbrace _{i\in\mathbb{I}} $ and
$ \left\lbrace V_{i}\right\rbrace _{i\in\mathbb{I}} $ are woven Riesz decomposition.
\end{proof}
\begin{corollary}
\label{C2}
If for every $ \sigma\subset\mathbb{I} $,
the family $ \left\lbrace W_{i}\right\rbrace _{i\in\sigma} \bigcup\left\lbrace V_{i}\right\rbrace _{i\in\sigma^{c}} $ is a orthonormal basis, Then
$ \left\lbrace W_{i}\right\rbrace _{i\in\mathbb{I}} $ and $ \left\lbrace V_{i}\right\rbrace _{i\in\mathbb{I}} $ are woven Riesz decomposition.
\end{corollary}
\begin{proof}
The result follows from Corollary \ref{C1} and the definition of woven of subspaces.
\end{proof}
From Corollary \ref{C2}, we have:
\begin{example}
Let $ \left\lbrace e_{i}\right\rbrace _{i\in\mathbb{I} } $ be an orthonormal basis for Hilbert space $ \mathcal{H} $ and define the subspaces:
$$ W_{1}=\overline{{\rm span}}\left\lbrace e_{2i}\right\rbrace _{i= 1}^{\infty} \quad and \quad
W_{2}=\overline{{\rm span}}\left\lbrace e_{2i-1}\right\rbrace _{i= 1}^{\infty}.$$
Since $ \mathcal{H}=W_{1}\oplus W_{2} $, then $\left\lbrace W_{1}, W_{2}\right\rbrace $ is a Parseval fusion frame for $ \mathcal{H} $, with respect to
$ \left\lbrace \nu_{i}\right\rbrace _{i\in\mathbb{I} } $ such that for every $ i\in\mathbb{I} $, $ \nu_{i}=1 $:
\begin{eqnarray*}
\Vert P_{W_{1}}(f)\Vert^{2}+\Vert P_{W_{2}}(f)\Vert^{2}&=&
\sum_{i=1}^{\infty}\vert\left\langle P_{W_{1}}(f), e_{i}\right\rangle \vert^{2}
+\sum_{i=1}^{\infty}\vert\left\langle P_{W_{2}}(f), e_{i}\right\rangle \vert^{2}\\
&=&\sum_{i=1}^{\infty}\vert\left\langle f, e_{2i}\right\rangle \vert^{2}
+ \sum_{i=1}^{\infty}\vert\left\langle f, e_{2i-1}\right\rangle \vert^{2}\\
&=&\sum_{i=1}^{\infty}\vert\left\langle f, e_{i}\right\rangle \vert^{2}\\&=&\Vert f\Vert^{2}.
\end{eqnarray*}
For constant $ \delta >0 $, we define the subspaces $ V_{1} $ and $ V_{2} $:
$$ V_{1}=\overline{{\rm span}}\left\lbrace\delta e_{2i-1}\right\rbrace _{i\in 1}^{\infty} \quad and \quad
V_{1}=\overline{{\rm span}}\left\lbrace\delta e_{2i}\right\rbrace _{i\in 1}^{\infty}.$$
Similarly $\left\lbrace V_{1}, V_{2}\right\rbrace $ is a tight fusion frame with bound $\delta $.
Besides, both of these fusion frames are orthonormal basis for $ \mathcal{H} $. So both of them are Riesz decomposition. Now from Corollary \ref{C2}, $ \left\lbrace W_{i}\right\rbrace _{i=1}^{2} $ and $ \left\lbrace V_{i}\right\rbrace _{i=1}^{2} $ are woven Riesz decomposition.
\end{example}



\end{document}